\documentclass[reqno]{amsart}
\usepackage{bbding}
\usepackage{mathrsfs}
\usepackage{mathrsfs}
\usepackage{amsfonts}
\usepackage{cases}
\usepackage{latexsym}
\usepackage{amsmath}
\usepackage[all]{xy}
\usepackage{stmaryrd}
\usepackage{amsmath,amssymb,amscd,bbm,amsthm,mathrsfs,dsfont}
\usepackage{color}
\usepackage[notref,notcite]{showkeys}
\usepackage{fancyhdr}
\usepackage{amsxtra,ifthen}
\usepackage{verbatim}

\usepackage{fancyhdr}
\usepackage{amsxtra,ifthen}
\usepackage{verbatim}

\usepackage[numbers,sort&compress]{natbib}
\usepackage[pagebackref]{hyperref}
\usepackage{hypernat}

\usepackage{enumitem}

\newtheorem{theorem}{Theorem}[section]
\newtheorem{lem}[theorem]{Lemma}
\newtheorem{prop}[theorem]{Proposition}
\newtheorem{coro}[theorem]{Corollary}
\theoremstyle{definition}
\newtheorem{rem}[theorem]{Remark}

\newcommand{\A}{\mathcal{A}}

\newcommand{\B}{\mathcal{B}}
\newcommand{\Z}{\mathcal{Z}}

\newcommand{\N}{\mathcal{N}}
\renewcommand{\l}{\lambda}
\renewcommand{\L}{\mathit{\Lambda}}

\newcommand{\F}{\Phi}
\newcommand{\sneq}{\subsetneqq}

\newcommand{\f}{\varphi}
\newcommand{\I}{\mathcal{I}}
\renewcommand{\i}{I}
\newcommand{\s}{\sigma}
\newcommand{\m}{\mu}
\renewcommand{\a}{\alpha}
\renewcommand{\b}{\beta}

\newcommand{\alg}{\mathrm{Alg}\,}

\newcommand{\ol}{\overline}
\newcommand{\ot}{\otimes}
\newcommand{\M}{\mathcal{M}}
\newcommand{\Nn}{\mathbb{N}}
\newcommand{\wt}{\widetilde}
\newcommand{\seq}{\subseteq}
\newcommand{\cdim}{\mathrm{codim}\,}

\newcommand{\lspan}{\mathrm{span}\,}

\newcommand{\Xb}{\mathcal{X}^\bot_-}
\newcommand{\0}{\mathbf 0}
\newcommand{\p}{\psi}
\newcommand{\Xm}{\mathcal{X}_-}

\newcommand{\sm}{\setminus}

\begin{document}

\title[Lie-Type Derivations of Nest Algebras on Banach Spaces]
{Lie-Type Derivations of Nest Algebras on Banach Spaces}

\author{Yuhao Zhang and Feng Wei}

\address{Zhang: School of Mathematics and Statistics, Beijing Institute of
Technology, Beijing, 100081, P. R. China}

\email{$2008_{-}{\rm joseph}$@sina.com}

\address{Wei: School of Mathematics and Statistics, Beijing Institute of
Technology, Beijing, 100081, P. R. China}

\email{daoshuo@hotmail.com}\email{daoshuowei@gmail.com}

\begin{abstract}
Let $\mathcal{X}$ be a Banach space over the complex field $\mathbb{C}$ and $\mathcal{B(X)}$ be the algebra of all bounded linear operators on $\mathcal{X}$. Let $\mathcal{N}$ be a non-trivial nest on $\mathcal{X}$, ${\rm Alg}\mathcal{N}$ be the nest algebra associated with $\mathcal{N}$, and $L\colon {\rm Alg}\mathcal{N}\longrightarrow \mathcal{B(X)}$ be a linear mapping. Suppose that $p_n(x_1,x_2,\cdots,x_n)$ is an $(n-1)$-th commutator
defined by $n$ indeterminates $x_1, x_2, \cdots, x_n$. It is shown that $L$ satisfies the rule 
$$
L(p_n(A_1, A_2, \cdots, A_n))=\sum_{k=1}^{n}p_n(A_1, \cdots, A_{k-1}, L(A_k), A_{k+1}, \cdots, A_n)
$$
for all $A_1, A_2, \cdots, A_n\in {\rm Alg}\mathcal{N}$ if and only if there exist a linear derivation
$D\colon {\rm Alg}\mathcal{N}\longrightarrow \mathcal{B(X)}$ and a linear mapping $H\colon {\rm Alg}\mathcal{N}\longrightarrow \mathbb{C}I$ vanishing
on each $(n-1)$-th commutator $p_n(A_1,A_2,\cdots, A_n)$ for all $A_1, A_2, \cdots, A_n\in {\rm Alg}\mathcal{N}$ such that $L(A)=D(A)+H(A)$ for all $A\in {\rm Alg}\mathcal{N}$.
\end{abstract}




\date{\today}

\maketitle

\section{Introduction}
\label{xxsec1}

Let $\mathcal{A}$ be an associative algebra over the complex field $\mathbb{C}$, $\mathcal{M}$ be an $(\mathcal{A}, \mathcal{A})$-bimodule and   
$\mathcal{Z}_{\mathcal{A}}(\mathcal{M})$ be the center of $\mathcal{M}$ relative to $\mathcal{A}$, where $\mathcal{Z}_{\mathcal{A}}(\mathcal{M}):=\{m\in \mathcal{M}: a\cdot m=m\cdot a\ \ \text{for all}\ \ a\in \mathcal{A}\}$. We shall write just $\mathcal{Z(A)}$ to denote $\mathcal{Z}_{\mathcal{A}}(\mathcal{A})$. A linear mapping $d: \mathcal{A} \longrightarrow \mathcal{M}$ is called an \textit{{\rm (}associative{\rm )} derivation} if $d(xy)=d(x)y+x d(y)$ holds true for all $x, y\in \mathcal{A}$. Let us denote the Lie (resp. Jordan) product of arbitrary elements $x, y\in \mathcal{A}$ by $[x, y]=xy-yx$ (resp. $x\circ y=xy+yx$). A \textit{Lie derivation} (resp. \textit{Lie triple derivation}) is a linear mapping $L: \mathcal{A} \longrightarrow \mathcal{M}$ which satisfies the derivation rule according to the Lie product, i.e.,
$$
\begin{aligned}
L([x, y])&=[L(x), y]+[x, L(y)] \\
(\text{resp.}\ \ L([[x, y],z])&=[[L(x), y],z]+[[x, L(y)],z]+[[x,y],L(z)])
\end{aligned}
$$
for all $x,y,z\in \mathcal{A}$. Recall that a linear mapping $\delta: \mathcal{A} \longrightarrow \mathcal{M}$ is a \textit{Jordan derivation} if $\delta(x\circ y)=\delta(x)\circ y+x\circ \delta(y)$ holds true for all $x, y\in \mathcal{A}$. Clearly, every derivation is a Lie
derivation as well as a Jordan derivation, and every Lie derivation is a Lie triple derivation. But, the converse statements are not true in general. It should be remarked that every Jordan derivation is also a Lie triple derivation, which is due to the formula $[[x, y], z]=x\circ (y\circ z)-y\circ (x\circ z)$ for all $x, y, z\in \mathcal{A}$. Consequently, the class of Lie triple derivations include both classes of Jordan derivations and Lie derivations simultaneously.

Inspired by the definitions of
Lie derivation and Lie triple derivation, one naturally expect to extend them in
one more general way. Suppose that $n\geq 2$ is a fixed positive
integer. Let us see a sequence of polynomials
$$
\begin{aligned}
p_1(x_1)&=x_1,\\
p_2(x_1,x_2)&=[p_1(x_1),x_2]=[x_1,x_2],\\
p_3(x_1,x_2,x_3)&=[p_2(x_1,x_2),x_3]=[[x_1,x_2],x_3],\\
p_4(x_1,x_2,x_3,x_4)&=[p_3(x_1,x_2,x_3),x_4]=[[[x_1,x_2],x_3],x_4],\\
 &\cdots\cdots,\\
p_n(x_1,x_2,\cdots,x_n)&=[p_{n-1}(x_1,x_2,\cdots,x_{n-1}),x_n].
\end{aligned}
$$
The polynomial $p_n(x_1,x_2,\cdots,x_n)$ is said to be an
$(n-1)$-\textit{th commutator} ($n\geq 2$).
A \textit{Lie $n$-derivation} is a $\mathbb{C}$-linear mapping $L:
\mathcal{A} \longrightarrow \mathcal{M}$ which satisfies the rule
$$
L(p_n(x_1,x_2,\cdots,x_n))=\sum_{k=1}^n
p_n(x_1,\cdots,x_{k-1}, L(x_k),x_{k+1},\cdots,x_n)
$$
for all $x_1,x_2,\cdots,x_n\in \mathcal{A}$. Lie $n$-derivations were introduced by Abdullaev
\cite{Abdullaev}, where the form of Lie $n$-derivations of a certain
von Neumann algebra (or of its skew-adjoint part) was described.
According to the definition, each Lie derivation is a Lie $2$-derivation and each Lie triple derivation is a Lie
$3$-derivation. Fo\v sner et al \cite{FosnerWeiXiao} showed that every Lie $n$-derivation
from $\mathcal{A}$ into $\mathcal{M}$ is a Lie $(n+k(n-1))$-derivation for each $k\in \Bbb{N}_0$.
Lie $2$-derivations, Lie $3$-derivations and Lie $n$-derivations are collectively referred to as \textit{Lie-type derivations}. Moreover, if 
$D\colon \mathcal{A}\longrightarrow \mathcal{M}$ is an associative derivation and $H\colon \mathcal{A}\longrightarrow \mathcal{Z}_{\mathcal{A}}(\mathcal{M})$ is a linear mapping vanishing on each $(n-1)$-th commutator $p_n(A_1, A_2, \cdots, A_n)$ for all $A_1, A_2, \cdots, A_n\in \mathcal{A}$, then the mapping $L=D+H$ is a Lie-type derivation. We
shall say that a Lie-type derivation is of \textit{standard form} in the case where it can be expressed in the
preceding form. Lie-type derivations and their standard decomposition problems in different backgrounds are extensively studied by a number of people, see \cite{ Abdullaev, BaiDu, Cheung, ForrestMarcoux, FosnerWeiXiao, Johnson, Lu, MathieuVillena, Miers2, Miers3, Qi, QiHou, SunMa, XiaoWei, ZhangWuCao}.

Nest algebras of operators on Hilbert space were introduced in 1965-1966 by
Ringrose \cite{Ringrose1, Ringrose2} and generalize in a certain way the set of all $n\times n$ matrices to infinite dimensions. There has been a great intrest in studying structure theory and various linear mappings of nest algebras on (Hilbert)Banach spaces, culminating in the elegant monograph by Davidson \cite{Davidson}. It was Marcoux and Sourour who initiated the study of Lie-type mappings on nest algebras in \cite{MarcouxSourour}. They extended the analysis of Lie automorphisms of non-selfadjoint operator algebras to the infinite dimensional setting by characterizing Lie isomorphisms of nest algebras. Let $\mathcal{H}$ be a complex, separable Hilbert space. Let ${\rm Alg}\mathcal{N}$ and ${\rm Alg}\mathcal{M}$ be the nest algebras associated with nests $\mathcal{N}$ and $\mathcal{M}$, respectively. A linear mapping $\Phi\colon {\rm Alg}\mathcal{N}\longrightarrow {\rm Alg}\mathcal{M}$ is a Lie isomorphism if and only if for all $A\in {\rm Alg}\mathcal{N}, \Phi(A)=\Psi(A)+h(A)I$, where $\Psi$ is an isomorphism or the negative of an anti-isomorphism and $h$ maps ${\rm Alg}\mathcal{N}$ into the center of ${\rm Alg}\mathcal{M}$ vanishing on all commutators. In the spirit of Marcoux and Sourour's work, Wang and Lu \cite{WangLu}, Qi et al \cite{QiHouDeng} independently obtained similar results for nest algebras on Banach spaces. In another direction, Christensen \cite{Christensen} proved every derivation on a nest algebra to be continuous and implemented. Spivack \cite{Spivack} developed the notion of nest
algebras to sets of operators on Banach space and proved that every continuous derivation on such a nest algebra is implemented by a bounded operator. Lu \cite{Lu} and Zhang et al \cite{ZhangWuCao} independently investigated Lie triple derivations of nest algebras on Hilbert spaces and observed that every Lie triple drivation has the standard form. Sun and Ma \cite {SunMa} generalized this result to nest algebras on Banach spaces by a series of complicated computations. Donsig et al \cite{DonsigForrestMarcoux} investigated derivations of semi-nest algebras and established a close tie between derivations and cohomology theory of semi-nest algebras.

Motivated by the above works, we will totally characterize all Lie-type derivations of nest algebras on Banach spaces in the current work. It turns out that every Lie-type derivation of nest algebras on Banach spaces is of standard form. It seems to us that this is the ultimate version of this line of characterization theorems in the case of nest algebra setting. It is clear that our main Theorem \ref{xxsec3.1} is not only one more common generalization of the afore-mentioned results, but our approach is quite different from that of existing works. For instance, Sun and Ma \cite{SunMa} derived their results by a tedious and complicated computational process. Adopting their methods to achieve our result, it would be an unbelievable task. Our approach is by far more conceptial, and reflects constructive skills strongly. This can be seen by the combinatorial applications of two crucial tools---rank one operators and the well-known Hanh-Banach theorem.

The outline of this article is organized as follows. We give some basic and necessary facts concerning on nest algebras in the second Section \ref{xxsec2}. The third Section \ref{xxsec3} is devoted to our main Theorem \ref{xxsec3.1}. Its proof will be realized by a systemic analysis for several different cases.

\section{Preliminaries}
\label{xxsec2}

Let $\mathcal{X, Y}$ be complex Banach spaces. Let us denote by $\mathcal{B(X,Y)}$ and ${\rm Hom}(\mathcal{X},\mathcal{Y})$ the Banach space of all bounded linear operators from $\mathcal{X}$ to $\mathcal{Y}$ and the linear space of all linear operators from $\mathcal{X}$ to $\mathcal{Y}$, respectively. We write $\mathcal{B(X)}:=\mathcal{B(X,X)}$ for the algebra of all boundend linear operators from $\mathcal{X}$ to itself, and $\mathcal{B(X)}$ has an identity, namely the identity operator $\i$. The \emph{topological dual space} $\mathcal{X}^*$ of $\mathcal{X}$ is defined to be $\mathcal{X}^*:=\mathcal{B}(\mathcal{X},\mathbb{C})$, i.e., the Banach space of all bounded linear functionals on $\mathcal{X}$. For arbitrary elements $x\in \mathcal{X}$ and $f\in \mathcal{X}^*$, one can define a \emph{rank-one operator} $x\ot f$ by $x\ot f(y)=f(y)x\;(\forall y\in \mathcal{X})$. $x\ot f$ is an element of $\mathcal{B(X)}$ and has the property $\|x\ot f\|=\|x\|\|f\|$. For any $T\in \mathcal{B(X,Y)}$, its \emph{adjoint operator} $T^*$ is refered to $T^*f(x)=f(Tx)\;(\forall f\in Y^*, x\in \mathcal{X})$. Then $*\colon \mathcal{B(X,Y)}\to\mathcal{B}(\mathcal{Y}^*, \mathcal{X}^*)$ is an isometric embedding. For an arbitrary subset $\mathcal{N}\seq \mathcal{X}$, we denote by $\mathcal{N}^\bot:=\{f\in \mathcal{X}^*\mid f(x)=0,\forall x\in \mathcal{N}\}$ the \emph{annihilator} of $\mathcal{N}$. For a subspace $\mathcal{V}\seq \mathcal{X}$, the \emph{codimension} $\cdim \mathcal{V}$ of $\mathcal{V}$ in $\mathcal{X}$ is defined as the dimension $\dim \mathcal{X/V}$ of the quotient space $\mathcal{X/V}$. Let us present two elementary results concerning dimension and codimension.

\begin{lem}\label{xxsec2.1}
Let $\mathcal{X}$ be a normed space, and $\mathcal{V}$ be a closed subspace of $\mathcal{X}$. If $\cdim \mathcal{V}>1$, then there exists $f\in \mathcal{X}^*\setminus\{0\}$ such that $\mathcal{V}+\lspan\{x\}\seq\ker f$ for all $x\in \mathcal{X}$. 
\end{lem}

\begin{proof}
For an arbitrary $x\in \mathcal{X}$, let us set $\mathcal{W}=\mathcal{V}+\lspan\{x\}$. Then $\mathcal{W}$ is closed, whch is due to the closedness of $\mathcal{V}$. In addition, $\cdim \mathcal{V}>1$ implies $\cdim \mathcal{W}\geq 1$, namely $\dim \mathcal{X/W}\geq 1$. By Hahn-Banach theorem, we know that there is $\wt f\in(\mathcal{X/W})^*\setminus\{0\}$. Since $\mathcal{W}$ is closed, the quotient mapping $\theta\colon \mathcal{X}\to \mathcal{X/W}, x\mapsto x+\mathcal{W}$ is continuous \cite[Theorem 1.5.8]{KadisonRingrose1}. Let us write $f=\wt f\circ \theta$. And then $f\in \mathcal{W}^\bot$ and $f$ is non-vanishing. 
\end{proof}

\begin{lem}\label{xxsec2.2}
Let $\mathcal{X}$ be a linear space over the complex field $\mathbb{C}$. Suppose that $\mathcal{V}$ and $\mathcal{W}$ are subspaces of $\mathcal{X}$ with $\cdim \mathcal{V}=\cdim \mathcal{W}<\infty$. If $\mathcal{V}\seq \mathcal{W}$, then $\mathcal{V}=\mathcal{W}$.
\end{lem}
\begin{proof}
In view of the fact $\mathcal{X/W}\cong (\mathcal{X/V})/(\mathcal{W/V})$, 
we have
\[
\infty>\dim \mathcal{W/V}=\cdim \mathcal{V}-\cdim \mathcal{W}=0,
\]
which entails $\mathcal{W}=\mathcal{V}$.
\end{proof}

Let $\mathcal{X}$ be a Banach space over the complex field $\mathbb{C}$, $\mathcal{B(X)}$ be the algebra of all bounded linear operators on $\mathcal{X}$. Let $\Lambda$ be an index set. A \textit{nest}
is a set $\mathcal{N}=\{N_\lambda\}_{\lambda\in \Lambda}$ of closed subspaces of $\mathcal{X}$
satisfying the following conditions:
\begin{enumerate}
\item[(1)] $\{0\}, \mathcal{X}\in \mathcal{N}$;
\item[(2)] If $N_1, N_2\in \mathcal{N}$, then
either $N_1\subseteq N_2$ or $N_2\subseteq N_1$;
\item[(3)] If $\{N_\lambda\}_{\lambda\in \Lambda}\subseteq \mathcal{N}$,
then $\bigcap_{\lambda\in \Lambda}N_\lambda\in \mathcal{N}$;
\item[(4)] If $\{N_\lambda\}_{\lambda\in \Lambda}\subseteq \mathcal{N}$, then the
norm closure of the linear span of $\bigcup_{i\in \Lambda} N_\lambda$ also lies
in $\mathcal{N}$.
\end{enumerate}
If $\mathcal{N}=\{\{0\}, \mathcal{X}\}$, then $\mathcal{N}$ is called a
\textit{trivial nest}, otherwise it is called a \textit{non-trivial nest}. As usual, we set  $\bigwedge_{\l\in\Lambda}N_\l:=\bigcap_{\l\in\Lambda}N_\l$, $\bigvee_{\l\in\Lambda}N_\l:=\ol{\lspan\bigcup_{\l\in\Lambda}N_\l}$.

The nest algebra associated with the nest $\mathcal{N}$, denoted by
${\rm Alg}\mathcal{N}$, is the weakly closed operator algebra
consisting of all operators that leave $\mathcal{N}$ invariant,
i.e.,
$$
{\rm Alg}\mathcal{N}=\{\ T\in
\mathcal{B(X)}\ | \ TN\subseteq N \ \text{for all}\  N\in
\mathcal{N}\ \}.
$$
It should be remarked that every finite dimensional nest algebra
is isomorphic to a real of complex block upper triangular matrix
algebra. If $\mathcal{N}$ is trivial, then
${\rm Alg}\mathcal{N}=\mathcal{B(X)}$, which is a prime algebra over
the real or complex field $\mathbb{C}$. In this paper we only
consider the nontrivial nest algebras.

Let $\mathcal{N}$ be a nest on a complex Banach space
$\mathcal{X}$ such that there exists a $N\in \mathcal{N}$
complemented in $\mathcal{X}$ and $\text{Alg}\mathcal{N}$ be the
nest algebra associated with $\mathcal{N}$. Then $\text{Alg}\mathcal{N}$ is a
triangular algebra over $\Bbb{C}$. Indeed, Since $N\in
\mathcal{N}$ is complemented in $\mathcal{X}$, there is a bounded
idempotent operator $P$ with range $N$. It is easy to check that
$P\in \text{Alg}\mathcal{N}$. Let us denote $M=(I-P)(\mathcal{X})$,
and let $\mathcal{A}=P\text{Alg}\mathcal{N}|_N,
\mathcal{M}=P\text{Alg}\mathcal{N}|_M$ and
$\mathcal{B}=(I-P)\text{Alg}\mathcal{N}|_M$. Then
$_{\mathcal{A}}\mathcal{M}_{\mathcal{B}}$ is faithful as left
$\mathcal{A}$-module and right $\mathcal{B}$-module. We therefore say that 
$$
{\rm Alg}\mathcal{N}=\left[
\begin{array}
[c]{cc}%
\mathcal{A} & \mathcal{M}\\
O & \mathcal{B}\\
\end{array}
\right].
$$
Note that ${\rm Alg}\mathcal{N}$ is a central algebra over its center $\mathcal{Z}({\rm Alg}\mathcal{N})=\Bbb{C}I$.

In particular, if $\mathcal{X}$ is a Hilbert space, then every nontrivial nest algebra is a
triangular algebra. Indeed, if $N\in \mathcal{N}\backslash \{\{0\}, \mathcal{X}\}$
and $E$ is the orthogonal projection onto $N$, then
$\mathcal{N}_1=E(\mathcal{N})$ and
$\mathcal{N}_2=(1-E)(\mathcal{N})$ are nests of $N$ and $N^{\perp}$,
respectively. Moreover,
${\rm Alg}\mathcal{N}_1=E{\rm Alg}\mathcal{N}E,
{\rm Alg}\mathcal{N}_2=(1-E){\rm Alg}\mathcal{N}(1-E)$ are
nest algebras and
$$
{\rm Alg}\mathcal{N}=\left[
\begin{array}
[c]{cc}%
{\rm Alg}\mathcal{N}_1 & E{\rm Alg}\mathcal{N}(1-E)\\
O & {\rm Alg}\mathcal{N}_2\\
\end{array}
\right] .
$$
However, it is not always the case for a nest $\mathcal{N}$ on a
general Banach space $\mathcal{X}$, since $N\in \mathcal{N}$ may be not
complemented. We refer the reader to \cite{Davidson} for the theory
of nest algebras.

For a Banach space $\mathcal{X}$, if $\N$ is a non-trivial nest on $\mathcal{X}$, then there exist an $N\in\N\setminus\{\{0\}, \mathcal{X}\}$ and a non-vanishing $f\in N^\bot$ (By Hahn-Banach theorem we get a non-vanishing $\wt f\in (\mathcal{X}/N)^*$, where $N$ is a closed subspace. And hence the mapping $g\colon \mathcal{X}\to \mathcal{X}/N$ is continuous. It suffices to take $f=\wt f\circ g$). We therefore have $\mathcal{X}/\ker f\cong\mathbb{C}$. By the fact $\dim\ker f\geq 1$, we know that the dimension of $\mathcal{X}$ is no less than $2$. For an arbitrary $N\in\N$, one can define
\[
N_-:=\bigvee\{M\in\N\colon M\sneq N\}\;(N\neq\{0\}),\ 
N_+:=\bigwedge\{M\in\N\colon N\sneq M\}\;(N\neq \mathcal{X}),
\]
and we set $\{0\}_-:=\{0\}$, $\mathcal{X}_+:=\mathcal{X}$.

\begin{lem}\emph{(\cite[Lemma~2.1]{SunMa})}\label{xxsec2.3}
Let $\N$ be a nest on a Banach space $X$. Then $x\ot f\in\alg\N$ if and only if there exists $N\in\N$ such that $x\in N$ and $f\in N_-^\bot:=(N_-)^\bot$ and equivalently, if and only if there exists $N\in\N$ such that $x\in N_+$ and $f\in N^\bot$.
\end{lem}

Let $\mathcal{A}$ be an associative algebra over the complex field $\mathbb{C}$, $\mathcal{M}$ be an $(\mathcal{A}, \mathcal{A})$-bimodule and   
$\mathcal{Z}_{\mathcal{A}}(\mathcal{M})$ be the center of $\mathcal{M}$ relative to $\mathcal{A}$.

\begin{prop}\emph{(\cite[Proposition 1.1]{FosnerWeiXiao})}\label{xxsec2.4}
Let $L\colon\A\to\M$ be a Lie $n$-derivation. Then for each $k\in\Nn$, $L$ is also a Lie $n+k(n-1)$-derivation. In particular, if $L$ a derivation, then $L$ is also a Lie $n$-derivation. 
\end{prop}

\begin{prop}\label{xxsec2.5}
Let $D\colon\A\to\M$ be a derivation and $H\colon\A\to\Z_\A(\M)$ be a linear mappping. Then $D+H$ is a Lie $n$-derivation if and only if $H$ is a linear mapping vanishing on each $(n-1)$-th commutator $p_n(A_1,\cdots,A_n)$ for all $A_1, A_2, \cdots, A_n\in \mathcal{A}$.
\end{prop}

\begin{proof}
The sufficiency is straightforward and we only prove the necessity. If $L=D+H$ is a Lie $n$-derivation, then $H=L-D$ is a linear mapping. It follows from  Proposition \ref{xxsec2.4} that 
\begin{align*}
L(p_n(A_1,\cdots,A_n))=&\sum_{i=1}^np_n(A_1,\cdots,\mathop{L(A_i)} _{\substack{\uparrow\\ i\text{-th}}},\cdots,A_n)\\
=&\sum_{i=1}^np_n(A_1,\cdots,\mathop{D(A_i)} _{\substack{\uparrow\\ i\text{-th}}},\cdots,A_n)\\
=&D(p_n(A_1,\cdots,A_n))
\end{align*}
for all $A_1,\cdots,A_n\in\A$. And hence $H(p_n(A_1,\cdots,A_n))=0$. 
\end{proof}

The following lemma is one common generalization of \cite[Problem 230]{Halmos}. 

\begin{lem}\label{xxsec2.6}
Let $A,B\in\mathcal{B(X)}$, $\l\in\mathbb{C}$. If $[A,B]=\l\i$, then $\l=0$.
\end{lem}

\begin{proof}
According to the assumption $[A,B]=\l\i$, we see that
\begin{align*}
A^2B-BA^2=&A^2B-ABA+ABA-BA^2\\
=&A(AB-BA)+(AB-BA)A=2\l A.
\end{align*}
A direct induction on $n$ shows that  $A^nB-BA^n=n\l A^{n-1}$. If the nilpotent index of $A$ is $n\in\Nn$, then $\l=0$. If $A$ is not nilpotent, then
\[
n\|A^{n-1}\||\l|\leqslant2\|A\|\|B\|\|A^{n-1}\|\;(\forall n\in\Nn).
\]
Thus $n|\l|\leqslant 2\|A\|\|B\|$, which implies $\l=0$.
\end{proof}

\begin{rem}\label{xxsec2.7}
From the above Lemma \ref{xxsec2.6}, we can say that if $n\geqslant 2$, $A_1,\cdots, A_n\in\mathcal{B(X)}$ and $p_n(A_1,\cdots,A_n)\in\mathbb{C}\i$, then $p_n(A_1,\cdots,A_n)=0$.
\end{rem}

Let $\mathcal{X}$ be a Banach space, $\N$ be a nest on $\mathcal{X}$ and ${\rm Alg}\mathcal{N}$ be the nest algebra associated with $\mathcal{N}$. Then $\mathcal{B(X)}$ can be considered as an $({\rm Alg}\mathcal{N}, {\rm Alg}\mathcal{N})$-bimodule.

\begin{lem}\emph{(\cite[Lemma 2.3]{SunMa})}\label{xxsec2.8}
\begin{enumerate}
\item[{\rm (1)}] If $\{0\}\neq N\in\N$ and $T\in\mathcal{B(X)}$ satisfy $TA=AT$ {\rm (}on $N${\rm )} for all $A\in\alg\N$, then there exists $\l\in\mathbb{C}$ such that $T=\l\i$ on $N$.
\item[{\rm (2)}] $\Z_{\alg\N}(\mathcal{B(X)})=\mathbb{C}\i$.
\end{enumerate}
\end{lem}

Taking into account Lemma \ref{xxsec2.6} and Lemma \ref{xxsec2.8} and using induction on $n$, it is not difficult to prove

\begin{lem}\label{xxsec2.9}
Let $\N$ be a nest on a Banach space $X$ and $L\colon\alg\N\to\mathcal{B(X)}$ be a Lie $n$-derivation. Then $L(\i)=\l\i$ for some $\l\in\mathbb{C}$. 
\end{lem}

\section{Lie-Type Derivations of Nest Algebras}
\label{xxsec3}

Let us first state our main theorem of this article. 

\begin{theorem}\label{xxsec3.1}
Let $\mathcal{X}$ be a Banach space, $\N$ be a non-trivial nest on $\mathcal{X}$ and ${\rm Alg}\mathcal{N}$ be the nest algebra associated with $\mathcal{N}$. Then every Lie-type derivaion from $\alg\N$ into $\mathcal{B(X)}$ is of standard form. 

More precisely, $L\colon\alg\N\to\mathcal{B(X)}$ is a Lie $n$-derivation if and only if there exist a derivation $D\colon\alg\N\to\mathcal{B(X)}$ and a linear mapping $H\colon\alg\N\to\mathbb{C}\i$ vanishing on each $(n-1)$-th commutator on $\alg\N$ such that $L=D+H$. In particular, if $L$ is continuous, then $D$ and $H$ are continuous as well. 
\end{theorem}

As a matter of fact, we will investigate more general cases, i.e., on certain subalgebras of $\alg\N$. For reading convenience, we have to introduce some notations. Let $\N$ be a non-trivial nest on a Banach space $\mathcal{X}$, and $\A$ be a subalgebra of $\alg\N$. Define
$$
\mathcal{X}(\A):=\{x\in \mathcal{X}\mid\exists f\in \mathcal{X}^*\sm\{0\},x\ot f\in\A\},
$$
and
$$
\mathcal{X}^\bot_-(\A):=\{f\in \mathcal{X}^\bot_-\mid\exists x\in \mathcal{X}\sm\{0\}, x\ot f\in\A\}.
$$

We put forward the following conditions.

\begin{enumerate}[fullwidth,itemindent=1em,label=($\spadesuit$\arabic*)]
\item $\mathcal{X}^\bot_-(\A)\neq\{0\}$ is a linear space.\label{spade.annih neq 0}
\item For each $f\in\mathcal{X}^\bot_-(\A)$, there exists $y\in \mathcal{X}\setminus \mathcal{X}_-$ such that $y\ot f\in\A$.\label{spade.y ot f in A}
\item The set of all rank-one operators in $\A$ contains $\{x\ot f\mid x\in \mathcal{X}(\A),f\in\mathcal{X}^\bot_-(\A)\}$.\label{spade.all rank 1 op in A}
\item For each $f_1,f_2\in\mathcal{X}^\bot_-(\A)$, $\dim \mathcal{X}(\A)\cap\ker f_i>1\,(i=1,2)$ and $\mathcal{X}(\A)\cap\ker f_1\cap\ker f_2\neq\{0\}$ hold true.\label{spade.X(A)inters kerf}
\item $\mathcal{X}(\A)=\mathcal{X}$.\label{spade.X(A)=X}
\item For any $A\in\A$, $x\in \mathcal{X}(\A)$ and $f\in\mathcal{X}^\bot_-(\A)$, $Ax\in \mathcal{X}(\A)$ and $A^*f\in\mathcal{X}^\bot_-(\A)$ hold true.\label{spade.Ax and A^*f}
\end{enumerate}

It is not difficult to figure out that not all subalgebras satisfy the above conditions, while there are certain subalgebras satisfying all these conditions (e.g. $\alg\N$ itself). \ref{spade.all rank 1 op in A} implies that $\mathcal{X}(\A)$ and $\mathcal{X}^\bot_-(\A)$ are normed linear spaces which inherite the norm of $\mathcal{X}$ and that of $\mathcal{X}^*$, respectively. The condition \ref{spade.X(A)inters kerf} implies that the dimension $\mathcal{X}$ is greater than $2$. As you see, we will address the standard form question of Lie-type derivations of a subalgebra $\A$ satisfying some of the above conditions in the next several subsections.

By Proposition \ref{xxsec2.5} and Lemma \ref{xxsec2.8}, the sufficiency of our main Theorem \ref{xxsec3.1} is clear. We shall give the proof of its necessity by a systemic analysis for several different cases.

\subsection{The case of $\dim\mathcal{X}^\bot_-=1$}
\ \\

Under this case, one can loosen the condition of Lemma~\ref{xxsec2.8} slightly as follows.

\begin{lem}\label{xxsec3.2}
Let $\N$ be a nest on a Banach space $\mathcal{X}$ with $\dim\mathcal{X}^\bot_-=1$. If $\A$ is a subalgebra of $\alg\N$ satisfying the conditions \emph{\ref{spade.annih neq 0}} and \emph{\ref{spade.X(A)=X}}, then $\Z_\A(\mathcal{B(X)})=\mathbb{C}\i$.
\end{lem}

\begin{proof}
By $\dim\mathcal{X}^\bot_-=1$ we know that $\mathcal{X}_-\sneq \mathcal{X}$. The condition \ref{spade.annih neq 0} implies that $\mathcal{X}^\bot_-=\mathcal{X}^\bot_-(\A)$. For any $y\in \mathcal{X}\sm\mathcal{X}_-$, by Hahn-Banach theorem (pick $\wt f\in(\mathcal{X}/\mathcal{X}_-)^*$ such that $\wt f(y+\mathcal{X}_-)=1$, and let $f=\wt f\circ \theta$, where $\theta$ is the quotient mapping), it follows that there exists $f\in\mathcal{X}^\bot_-$ such that $f(y)=1$. Since $\dim\mathcal{X}^\bot_-(\A)=1$ and $\mathcal{X}(\A)=\mathcal{X}$, we get $x\ot f\in\A$ $(\forall x\in \mathcal{X})$. Thus for any $A\in\mathcal{Z_A}(\mathcal{B(X)})$, we have $A(x\ot f)y=(x\ot f)Ay$. That is, $Ax=f(Ay)x$. By the arbitrariness of $x$, we assert that $A\in\mathbb{C}\i$.
\end{proof}

\begin{theorem}\label{xxsec3.3}
Let $\N$ be a non-trivial nest of a Banach space $\mathcal{X}$ with $\dim\mathcal{X}^\bot_-=1$. Suppose that $\A$ is a subalgebra of $\alg\N$ satisfying \emph{\ref{spade.annih neq 0}, \ref{spade.y ot f in A}}, $\mathcal{X}_-\seq \mathcal{X}(\A)$ and $\i\in\A$. Then each Lie-type derivation from ${\rm Alg}\mathcal{N}$ into $\mathcal{B(X)}$ is of standard form. Precisely speaking, for any Lie $n$-derivation $L\colon\A\to\mathcal{B(X)}$, there exists a derivation $D\colon \A\to\mathcal{B(X)}$ and a linear mapping $H\colon\A\to\mathbb{C}\i$ vanishing on all $(n-1)$-th commutators of $\A$ such that $L=D+H$. Moreover, if $L$ is continuous, then $D$ and $H$ are continuous.
\end{theorem}

\begin{rem}\label{xxsec3.4}
The conditions what the subalgebra $\A$ satisfies amount to saying \ref{spade.annih neq 0}, $\mathcal{X}(\A)=\mathcal{X}$ and $\i\in\A$.
\end{rem}

\begin{proof}
By $\dim\mathcal{X}^\bot_-=1$ and \ref{spade.annih neq 0} it follows that $\mathcal{X}^\bot_-=\mathcal{X}^\bot_-(\A)$. If $\cdim \mathcal{X}_->1$, then there are linear independent elements $x_1+\mathcal{X}_-$ and $x_2+\mathcal{X}_-$ in $\mathcal{X}/\mathcal{X}_-$. By an analogous proof of Lemma~\ref{xxsec2.1}, one can construct $f_1,f_2\in\mathcal{X}^\bot_-$ satisfying $f_i(x_j)=\delta_{ij}$, where $\delta_{ij}$ is the Kronecker sign. Then $f_1, f_2$ are linear independent, which contradicts with $\dim\mathcal{X}^\bot_-=1$. Thus we assert $\cdim \mathcal{X}_-=1$. Basing on this fact, we know that for any $f\in\mathcal{X}^\bot_-\sm\{0\}$, $\ker f=\mathcal{X}_-$ holds true. Fixing certain $f_0\in\mathcal{X}^\bot_-(\A)\setminus\{0\}=\mathcal{X}^\bot_-\setminus\{0\}$, by \ref{spade.y ot f in A}, there exists $y\in \mathcal{X}\setminus \mathcal{X}^-$ such that $y\ot f_0\in\A$. Denote $M=\lspan\{y\}$, and then $\mathcal{X}=\mathcal{X}_-\oplus M$. Indeed, since $\s\colon M\to \mathcal{X}/\mathcal{X}_-,\,y\mapsto y+\mathcal{X}_-$ is an isomorphism, and then for any $f\in\mathcal{X}^\bot_-\sm\{0\}$, $\s_f\colon \mathcal{X}/\mathcal{X}_-\to\mathbb{C},\,y+\mathcal{X}_-\mapsto f(y)$ is also an isomorphism by fundamental theorem of homomorphism. We therefore have  $M\overset{\s_f\circ\s}{\cong}\mathbb{C}$. Let us take $x\in \mathcal{X}$ and $f\in\mathcal{X}^\bot_-\sm\{0\}$. Then there exists $y\in M$ with $f(y)=\s_f\circ\s(y)=f(x)$. Henceforth $x-y\in\ker f=\mathcal{X}_-$, i.e., $\mathcal{X}\seq\mathcal{X}_-+M$.

Let us choose $x_0\in M$ such that $f_0(x_0)=1$, and define $P=x_0\ot f_0\in\A$ and set $Q=\i-P$. Then $P$ and $Q$ are idempotents satisfying $PQ=QP=0$, $P\mathcal{X}=M$ and $Q\mathcal{X}=\mathcal{X}_-$. We get the Peirce decomposition $\mathcal{B(X)}=P\mathcal{B(X)}P\oplus P\mathcal{B(X)}Q\oplus Q\mathcal{B(X)}P\oplus Q\mathcal{B(X)}Q$. For any $A\in\A$, since $\dim\mathcal{X}^\bot_-=1$ and $A^*f_0\in\mathcal{X}^\bot_-=\mathcal{X}^\bot_-(\A)$, there exists $\l_A\in\mathbb{C}$ such that $A^*f_0=\l_Af_0$. Therefore
\begin{align*}
PAQ=(x_0\ot f_0)AQ=(x_0\ot A^*f_0)Q=\l_Ax_0\ot f_0Q=\l_APQ=0.
\end{align*}
Thus $\A=\A_{11}\oplus\A_{21}\oplus\A_{22}$, where $\A_{11}=P\A P$, $\A_{21}=Q\A P$, $\A_{22}=Q\A Q$.

We assert that there is a derivation $D\colon\A\to\mathcal{B(X)}$ such that the Lie $n$-derivation $\wt L=L-D$ satisfy the following conditions: if $A\in\A_{ii}\,(i=1,2)$, then $\wt L(A)-H_{ii}(A)\i\in\B_{ii}$, where $H_{ii}\colon\A_{ii}\to\mathbb{C}$ is a linear mapping, and $\B_{11}=P\mathcal{B(X)}P$, $\B_{22}=Q\mathcal{B(X)}Q$; if $A\in\A_{21}$, then $\wt L(A)\in\B_{21}:=Q\mathcal{B(X)}P$. Let us prove this assertion in two cases.

\emph{Case 1: $n$ is odd.} For an arbitrary $A\in\A_{21}$, we have $AQ=0$ and $QA=A$. A direct computation shows that
$$
\begin{aligned}
L(A)=&L(p_n(A,Q,\cdots,Q))\\
=&p_n(L(A),Q,\cdots,Q)+\sum_{i=2}^{n-1}p_n(A,\cdots,Q,\mathop{L(Q)} _{\substack{\uparrow\\ i\text{-th}}},Q, \cdots,Q)\\
&+p_n(A,\underbrace{Q,\cdots,Q}_{\text{odd numbers}}, L(Q))\\
=&QL(A)P+PL(A)Q-(n-2)(AL(Q)-AL(Q)Q-QL(Q)A)\\
&-AL(Q)+L(Q)A.
\end{aligned}\eqno{(3.1)}
$$
Multiplying $Q$ (resp. $P$) from the left (resp. right) side of (3.1), we obtain
\[
QL(Q)A=AL(Q)P.
\]
For any $x\in\mathcal{X}_-$, since $\mathcal{X}_-\seq \mathcal{X}(\A)$ and $\mathcal{X}^\bot_-(\A)=\mathcal{X}^\bot_-$ is $1$-dimensional, we thus get $A_0=x\ot f_0\in\A$. Obviously, $A_0=QA_0P\in\A_{21}$ and hence
$$
QL(Q)A_0=A_0L(Q)P=(x\ot f_0)L(Q)(x_0\ot f_0)=f_0(L(Q)x_0)(x\ot f_0).\eqno{(3.2)}
$$
Using (3.2) to act on $x_0$, we have
$$
QL(Q)x=f_0(L(Q)x_0)x.
$$
The arbitrariness of $x\in\Xm$ implies that
$$
QL(Q)Q=f_0(L(Q)x_0)Q. \eqno{(3.3)}
$$
A direct verification gives 
$$
PL(Q)P=(x_0\ot f_0)L(Q)P=f_0(L(Q)x_0)P. \eqno{(3.4)}
$$
Combining (3.3) with (3.4) we obtain
$$
QL(Q)Q+PL(Q)P=f_0(L(Q)x_0)\i.
$$
Let us write $T=PL(Q)Q-QL(Q)P$, and define $D_T(\cdot)=[T,\cdot\;]$. Then $D_T$ is a (continuous) derivation of $\A$ satisfying the condition
\[
L(Q)=D_T(Q)+f_0(L(Q)x_0)\i.
\]
Let us put $\wt L=L-D_T$. Then $\wt L$ is a Lie $n$-derivation on $\A$ and $\wt L(Q)=f_0(L(Q)x_0)\i$. A straightforward computation yields that
\begin{align*}
\wt L(A)=&\wt L(p_n(A,Q,\cdots,Q))\\
=&p_n(\wt L(A),Q,\cdots, Q)+p_n(A,Q,\cdots,Q,\wt L(Q))\\
=&Q\wt L(A)P+P\wt L(A)Q.
\end{align*}
For any $A'\in\A_{21}$ and $B\in\A$, we get
\begin{align*}
0=&\wt L(p_n(A,A',B,\cdots,B))\\
=&p_n(\wt L(A),A',B,\cdots,B)+p_n(A,\wt L(A'),B,\cdots,B)\\
=&p_{n-1}([\wt L(A),A']+[A,\wt L(A')], B,\cdots,B).
\end{align*}
By Lemma~\ref{xxsec2.6}, Lemma~\ref{xxsec3.2} and a recursive computation, we arrive at $[\wt L(A), A']+[A,\wt L(A')]=\l\i$, where $\l\in\mathbb{C}$. Thus
\begin{align*}
[Q\wt L(A)P+P\wt L(A)Q,A']=&[\wt L(A),A']\\
=&\l\i-[A,\wt L(A')]\\
=&\l\i-p_n(-A,Q,\cdots, Q,\wt L(A'))\\
=&\l\i-\wt L(p_n(-A,Q,\cdots,Q,A'))+p_n(\wt L(-A),Q,\cdots,Q,A')\\
=&\l\i-\wt L(p_2(A,A'))+p_3(Q,\wt L(A),A')\\
=&\l\i+p_3(Q,\wt L(A),A')\\
=&\l\i+p_3(Q,Q\wt L(A)P+P\wt L(A)Q,A')\\
=&\l\i+[Q\wt L(A)P-P\wt L(A)Q,A'].
\end{align*}
We therefore have $[P\wt L(A)Q, A']=\frac{1}{2}\l\i$. By Lemma~\ref{xxsec2.6}, it follows that $[P\wt L(A)Q, A']=0$. Multiplying by $P$ on the left side, we get $P\wt L(A)QA'=0$. In particular, $P\wt L(A)QA_0=0$. Considering its action on $x_0$, we obtain $P\wt L(A)Qx=0$. Note that $QM=\{0\}$, and thus $P\wt L(A)Q=0$. And hence $\wt L(A)=Q\wt L(A)P\in\B_{21}$.

For an arbitrary $A\in\A_{22}$, we have
\begin{align*}
0=\wt L(p_n(A,Q,\cdots,Q))=p_n(\wt L(A),Q,\cdots,Q)=Q\wt L(A)P+P\wt L(A)Q.
\end{align*}
So $0=Q\wt L(A)P=P\wt L(A)Q$. This gives that $\wt L(A)=Q\wt L(A)Q+P\wt L(A)P$. For any $B\in\A_{11}$ and $A'\in\A$, we get
\[
0=\wt L(p_n(A,B,A',\cdots,A'))=p_{n-1}([\wt L(A),B]+[A,\wt L(B)],A',\cdots,A').
\]
Taking into account Lemma~\ref{xxsec2.6} and Lemma~\ref{xxsec3.2} and using a recursive computation, we conclude that $[\wt L(A),B]+[A,\wt L(B)]=\m\i$, where $\m\in\mathbb{C}$. Since $B\in\A_{11}$, $[\wt L(A),B]=[P\wt L(A)P,B]$. Thus we obtain
\[
[P\wt L(A)P,B]=P([P\wt L(A)P,B]+[A,\wt L(B)])P=\m P.
\]
Looking on the above equation as an operator equation over $P\mathcal{X}=M$ and using Lemma \ref{xxsec2.6}, we see that $\m=0$. In particular, take $B=x\ot f_0\in\A_{11}$, where $x\in M$. Considering the action of $[P\wt L(A)P,x\ot f_0]=0$ on $x_0$, we get
\[
P\wt L(A)P=H_{22}(A)P,
\]
where $H_{22}(A)=f_0(P\wt L(A)x_0)\i\in\mathbb{C}\i$. The preceding equality shows the linearity and boundedness (if $L$ is bounded) of $H_{22}(A)$ for all $A\in\A_{22}$. Moreover, we have
\begin{align*}
\wt L(A)-H_{22}(A)=&Q\wt L(A)Q+P\wt L(A)P-H_{22}(A)(P+Q)\\
=&Q\wt L(A)Q-H_{22}(A)Q\in\B_{22}.
\end{align*}
By an analogous manner, one can prove that for any $A\in\A_{11}$, there exists a linear operator (and bounded if $L$ is bounded) $H_{11}\colon\A_{11}\to\mathbb{C}\i$ such that $\wt L(A)-H_{11}(A)\in\B_{11}$.

\indent\emph{Case 2: $n$ is even.} For an arbitrary $A\in\A_{21}$, a direct calculation shows that
$$
\begin{aligned}
-L(A)=&L(p_n(A,Q,\cdots,Q))\\
=&p_n(L(A),Q,\cdots,Q)+\sum_{i=2}^{n-1}p_n(A,\cdots,Q,\mathop{L(Q)} _{\substack{\uparrow\\ i\text{-th}}},Q, \cdots,Q)\\
&+p_n(A,\underbrace{Q,\cdots,Q}_{\text{even numbers}}, L(Q))\\
=&p_2(L(A),Q)-(n-2)p_3(A,L(Q),Q)+[A,L(Q)]\\
=&L(A)Q-QL(A)-(n-2)(AL(Q)Q-AL(Q)+QL(Q)A)\\
&+AL(Q)-L(Q)A.
\end{aligned}\eqno{(3.5)}
$$
Multiplying by $Q$ on the left side of (3.5) and by $P$ on the right side of (3.5), we obtain
\[
QL(Q)A=AL(Q)P.
\]
Adopting the same discussion as in \emph{Case 1}, we have
\[
L(Q)=[T,Q]=f_0(L(Q)x_0)\i.
\]
Denote $D_T(\cdot)=[T,\cdot\;]$ and $\wt L=L-D_T$. Then $\wt L(Q)=f_0(L(Q)x_0)\i$. Thus 
\begin{align*}
-\wt L(A)=&\wt L(p_n(A,Q\cdots,Q))\\
=&p_n(\wt L(A),Q,\cdots,Q)\\
=&p_2(\wt L(A),Q)\\
=&\wt L(A)Q-Q\wt L(A)\\
=&P\wt L(A)Q-Q\wt L(A)P.
\end{align*}
Therefore $\wt L(A)=Q\wt L(A)P-P\wt L(A)Q$. For an arbitrary $A'\in\A_{21}$ and $B\in\A$, a straightforward computation shows that
\begin{align*}
0=&\wt L(p_n(A,A',B,\cdots,B))\\
=&p_n(\wt L(A),A',B,\cdots,B)+p_n(A,\wt L(A'),B,\cdots,B)\\
=&p_{n-1}([\wt L(A),A']+[A,\wt L(A')],B,\cdots,B).
\end{align*}
By Lemma~\ref{xxsec2.6}, Lemma~\ref{xxsec3.2} and a recursive computation, we conclude
$$
[\wt L(A),A']+[A,\wt L(A')]=\l\i, \eqno{(3.6)}
$$
where $\l\in\mathbb{C}$. And hence
\begin{align*}
[\wt L(A),A']=&[Q\wt L(A)P-P\wt L(A)Q,A']\\
=&\l\i-[A,\wt L(A')]\\
=&\l\i-[p_{n-1}(A,Q,\cdots,Q),\wt L(A')]\\
=&\l\i-\wt L(p_n(A,Q,\cdots,Q,A'))+p_n(\wt L(A),Q,\cdots,Q,A')\\
=&\l\i-p_2(A,A')+[p_3(\wt L(A),Q,Q),A']\\
=&\l\i+[-P\wt L(A)Q+Q\wt L(A)P,A'].
\end{align*}
This forces $\l=0$. By (3.6) we get
$$
-P\wt L(A)A'+A'\wt L(A)Q+A\wt L(A')-\wt L(A')A=0.\eqno{(3.7)}
$$
Multiplying by $Q$ on the right side of (3.7), we obtain $(\wt L(A)A'+\wt L(A')A)Q=0$. Multiplying by $P$ on the left side of (3.7), we have $P(\wt L(A)A'+\wt L(A')A)=0$. Let us put $S=\wt L(A)A'+\wt L(A')A$. Then we see that 
\[
SQ=0,\;PS=0.
\]
On the other hand, since $\wt L(A')=Q\wt L(A')P-P\wt L(A')Q$, we by (3.6) arrive at
\[
-P\wt L(A)A'+A'\wt L(A)Q-A\wt L(A')Q+P\wt L(A')A=0.
\]
Thus
\[
\B_{11}\ni P(-\wt L(A)A'+\wt L(A')A)P=Q(A\wt L(A')-A'\wt L(A))Q\in\B_{22}.
\]
Therefore $P\wt L(A)A'=P\wt L(A')A$ and $A\wt L(A')Q=A'\wt L(A)Q$. Applying the fact $PS=0$ yields that $P\wt L(A)A'=0$. In particular, for $A'=x\ot f_0\in\A_{21}$, where $x\in\Xm$, the relation
\[
0=P\wt L(A)A'=P\wt L(A)Q(x\ot f_0)
\]
holds true. Considering its action on $x_0$, we observe that $P\wt L(A)Qx=0$. Since $QM=0$, we get $P\wt L(A)Q=0$ and hence $\wt L(A)=Q\wt L(A)P\in\B_{21}$.

\indent For an arbitrary $A\in\A_{22}$, a direct calculation shows that
$$
\begin{aligned}
0=&\wt L(p_n(A,Q,\cdots,Q))\\
=&p_n(\wt L(A),Q,\cdots,Q)\\
=&p_2(\wt L(A),Q)\\
=&P\wt L(A)Q-Q\wt L(A)P.
\end{aligned}
$$
This gives $Q\wt L(A)P=P\wt L(A)Q=0$. Thus $\wt L(A)=Q\wt L(A)Q+P\wt L(A)P$. The remainder proving process and the discussion for $A\in\A_{11}$ are identical with the counterparts of \emph{Case 1}, we do not repeat them here. Unitl now we complete the proof of the assertion. 

We now come back and continue to prove this theorem. For an arbitrary $A=A_1+A_2+A_3\in\A=\A_{11}\oplus \A_{21}\oplus\A_{22}$, we establish a mapping 
$$
\begin{aligned}
H\colon\A&\to\mathbb{C}\i\\
A&\mapsto (H_{11}(A_1)+H_{22}(A_3))\i.
\end{aligned}
$$
The linearity of $H_{ii}\,(i=1,2)$ implies that $H$ is linear. If $L$ is continuous, by properties of direct product, it is not difficult to verify the continuity of $H$. Define a mapping $\wt D\colon\A\to\mathcal{B(X)}$ via the relation $\wt D=\wt L-H$. Then $\wt D$ is a linear mapping such that $\wt D(A_1)\in\B_{11}$, $\wt D(A_2)\in\B_{21}$ and $\wt D(A_3)\in\B_{22}$.

We next illustrate that $\wt D$ is a derivation. In light of the fact 
\[
\A_{21}\ni A_3A_2=[A_3,A_2]=[[A_2,Q],A_3]=(-1)^{n-1}p_n(A_2,Q,\cdots,Q,A_3),
\]
we see that
\begin{align*}
\wt D(A_3A_2)=&\wt L(A_3A_2)\\
=&(-1)^{n-1}(p_n(\wt L(A_2),Q,\cdots,Q,A_3)+p_n(A_2,Q,\cdots,Q,\wt L(A_3)))\\
=&p_2(A_3,\wt L(A_2))+p_2(\wt L(A_3),A_2)\\
=&p_2(A_3,\wt D(A_2))+p_2(\wt D(A_3),A_2)\\
=&A_3\wt D(A_2)+\wt D(A_3)A_2.
\end{align*}
For an arbitrary $A_3'\in\A_{22}$, it follows that
$$
\wt D(A_3A_3'A_2)=\wt D((A_3A_3')A_2)=\wt D(A_3A_3')A_2+A_3A_3'\wt D(A_2).\eqno{(3.8)}
$$
On the other hand, we have
$$
\wt D(A_3A_3'A_2)=\wt D(A_3(A_3'A_2))=\wt D(A_3)A_3'A_2+A_3\wt D(A_3')A_2+A_3A_3'\wt D(A_2).\eqno{(3.9)}
$$
Comparing the relations (3.8) with (3.9) yields
$$
\wt D(A_3A_3')A_2=(\wt D(A_3)A_3'+A_3\wt D(A_3'))A_2.\eqno{(3.10)}
$$
In particular, take $A_2=x\ot f_0$, where $x\in\Xm$. Using (3.10) to act on  $x_0$, we obtain
\[
\wt D(A_3A_3')x=(\wt D(A_3)A_3'+A_3\wt D(A_3'))x.
\]
In view of the fact $QM=\{0\}$, we know that
\[
\wt D(A_3A_3')=\wt D(A_3)A_3'+A_3\wt D(A_3').
\]
For the case of
\[
\A_{21}\ni A_2A_1=[A_2,A_1]=(-1)^np_n(A_2,Q,\cdots,Q,A_1), 
\]
its discussion is parallel with that of case $A_3A_2$, and for the case of $A_1A_1'$ (where $A_1'\in\A_{11}$), its proving process is totally similar to the proof of the case $A_3A_3'$. 

For arbitrary elements $A=A_1+A_2+A_3$, $B=B_1+B_2+B_3\in\A=\A_{11}\oplus\A_{21}\oplus\A_{22}$, we eventually get
$$
\begin{aligned}
\wt D(AB)&=\wt D(A_1B_1+A_2B_1+A_3B_2+A_3B_3)\\
&=\wt D(A)B+A\wt D(B).
\end{aligned}
$$
This shows that $\wt D$ is a derivation. Let $D=\wt D+D_T$ ($D_T$ is the derivation established in the previous assertion). Then $D$ is a (continuous, if $L$ is continuous) derivation, and has the decomposition
\[
L=\wt L+D_T=\wt D+H+D_T=D+H.
\]
By Proposition~\ref{xxsec2.5}, we know that $H\colon\A\to\mathbb{C}\i$ vanishes on each $(n-1)$-th commutator on $\A$.
\end{proof}

In particular, if $\mathcal{X}$ is a $2$-dimensional Banach space, then a consequence of Theorem \ref{xxsec3.3} is

\begin{coro}\label{xxsec3.5}
Suppose that $\dim \mathcal{X}=2$ and that $\N=\{\{0\},N, \mathcal{X}\}$ is a non-trivial nest on $\mathcal{X}$. Then each Lie-type derivation from ${\rm Alg}\mathcal{N}$ into $\mathcal{B(X)}$ is of standard form. Precisely speaking, for any Lie $n$-derivation $L\colon{\rm Alg}\mathcal{N}\to\mathcal{B(X)}$, there exists a derivation $D\colon {\rm Alg}\mathcal{N}\to\mathcal{B(X)}$ and a linear mapping $H\colon{\rm Alg}\N\to\mathbb{C}\i$ vanishing on all $(n-1)$-th commutators of ${\rm Alg}\N$ such that $L=D+H$. Moreover, if $L$ is continuous, then $D$ and $H$ are both continuous.
\end{coro}

\subsection{The case of $\dim\Xb>1$}\label{sect.dim Xb>1}

\begin{lem}\label{xxsec3.6}
Let $\mathcal{X}$ be a Banach space with $\dim \mathcal{X}>2$, $\Xm\sneq \mathcal{X}$ and $\N$ be a non-trivial nest on $\mathcal{X}$. If a subalgebra $\A\seq\alg\N$ satisfies the conditions \emph{\ref{spade.annih neq 0}--\ref{spade.X(A)inters kerf}}, and $L\colon\A\to\mathcal{B(X)}$ is a Lie $n$-derivation, then there is a bilinear mapping $h\colon \mathcal{X}(\A)\times\Xb(\A)\to\mathbb{C}$ such that 
\[
(L(x\ot f)-h(x,f)\i)(\mathcal{X}(\A)\cap\ker f)\seq\mathbb{C}x, \ \ \forall x\in \mathcal{X}(\A), \ \ \forall f\in\Xb(\A),
\]
where
\[
h(x,f)=\left\{
\begin{array}{c l}
0&,\;x\in\ker f\\
\frac{1}{f(x)}f(L(x\ot f)x)&,\;\text{\emph{otherwise}}
\end{array}\right.
\]
In particular, for $x\in \mathcal{X}(\A)\cap\ker f$, we have $L(x\ot f)\ker f\seq\mathbb{C} x$. If $L$ is continuous, then $h$ is also continuous.
\end{lem}

\begin{proof}

The case of $f=0$ is trivial. We next assume that $f\in\Xb(\A)\sm\{0\}$. If $x\in\ker f$, by \ref{spade.y ot f in A}, there exists $y\in \mathcal{X}\sm\Xm$ such that $f(y)=1$ and $y\ot f\in\A$. Let us put $y'=y+x$. Then $y'\in \mathcal{X}(\A)$ and $y, y', x$ are linearly independent. For any $z\in\ker f$, a direct computation shows that
\begin{align*}
L(x&\,\ot\, f)z=L(p_n(x\ot f,y\ot f,\cdots,y\ot f))z\\
=&p_n(L(x\ot f),y\ot f,\cdots,y\ot f)z+p_n(x\ot f,y\ot f,\cdots,y\ot f,L(y\ot f))z\\
&+\sum_{i=2}^{n-1}p_n(x\ot f,y\ot f,\cdots,y\ot f,\mathop{L(y\ot f)}_{\substack{\uparrow\\ i\text{-th}}},y\ot f,\cdots,y\ot f)z\\
=&(-1)^{n-1}(y\ot f)L(x\ot f)z+p_2(x\ot f,L(y\ot f))z+0\\
=&(-1)^{n-1}f(L(x\ot f)z)y+f(L(y\ot f)z)x.
\end{align*}
By an analogous manner, we have
\[
L(x\ot f)z=(-1)^{n-1}f(L(x\ot f)z)y'+f(L(y'\ot f)z)x.
\]
The linear independence of $y,y',x$ implies that $f(L(x\ot f)z)=0$. We therefore have $L(x\ot f)z\in\mathbb{C} x$.

If $x\notin\ker f$, for an arbitrary $z\in \mathcal{X}(\A)\cap\ker f$, it follows that
$$
\begin{aligned}
\notag f(x)^{n-1}L(z\ot f)=&L(p_n(z\ot f,x\ot f,\cdots,x\ot f))\\
=&\sum_{i=1}^{n}p_n(z\ot f,x\ot f,\cdots,\mathop{L(x\ot f)}_{\substack{\uparrow\\ i\text{-th}}},\cdots,x\ot f).
\end{aligned}\eqno{(3.11)}
$$
By a direct verification, we get
\begin{align*}
p_n(z\ot f,x\ot f,\cdots,&\mathop{L(x\ot f)}_{\substack{\uparrow\\ i\text{-th}}},\cdots,x\ot f)\\
=&f(x)^{i-2}p_{n-(i-2)}(z\ot f,L(x\ot f),x\ot f,\cdots,x\ot f)\;(i>1)
\end{align*}
and
$$
p_n(z\ot f,x\ot f,\cdots,x\ot f,L(x\ot f))=f(x)^{n-2}p_2(z\ot f,L(x\ot f)).
$$
For an arbitrary $A_1\in\mathcal{B(X)}$, one can define $A_i:=p_i(A_1,x\ot f,\cdots,x\ot f)\;(i\in\Nn)$. By an induction on $i$, it follows that
$$
p_2(A_{k},x\ot f)x=f(x)^{k}A_1x-\sum_{i=2}^{k}f(x)^{k-i}f(A_ix)x,\ \  \forall k\in\Nn.\eqno{(3.12)}
$$
Considering the action of (3.11) on $x$ and using (3.12) we obtain
\begin{align*}
f(x)^{n-1}L(z\ot f)x=&f(x)^{n-1}L(z\ot f)x-\l_1x\\
&+(n-2)f(x)^{n-2}p_2(z\ot f,L(x\ot f))x-\l_2x\\
&+f(x)^{n-2}p_2(z\ot f,L(x\ot f))x,
\end{align*}
where $\l_1,\l_2\in\mathbb{C}$. Simplifying the above equality, we get
\[
L(x\ot f)z=\frac{1}{f(x)}f(L(x\ot f)x)z-\frac{\l_1+\l_2}{(n-1)f(x)^{n-1}}x.
\]
This implies that $(L(x\ot f)-h(x,f)\i)z\in\mathbb{C} x$.

We next show that $h(x,f)$ is bilinear. Fixing $f\in\Xb(\A)$, for any $x,y\in \mathcal{X}(\A)$ and $a,b\in\mathbb{C}$, if $x,y\in\ker f$, then this case is trivial. If $x\in\ker f$ while $y\notin\ker f$, by \ref{spade.X(A)inters kerf} we can choose $z\in \mathcal{X}(\A)\cap\ker f$ such that $x,y,z$ are linearly independent. Considering the actions of $L(x\ot f)$, $L(y\ot f)$, $L((ax+by)\ot f)$ on $z$, respectively, the linearity of $L$ entails
$$
(h(ax+by,f)-bh(y,f))z\in\mathbb{C} x+\mathbb{C} y.
$$
Thus $h(ax+by,f)=bh(y,f)$. If $x,y\notin\ker f$, and $x,y$ are linearly dependent, then we take $z\in \mathcal{X}(\A)\cap\ker f$ such that $z,x$ are linearly independent. Taking into account the actions of $L(x\ot f)$, $L(y\ot f)$, $L((ax+by)\ot f)$ on $z$, respectively, the linearity of $L$ implies 
$$
(h(ax+by,f)-ah(x,f)-bh(y,f))z\in\mathbb{C} x.
$$
And hence $h(ax+by,f)=ah(x,f)+bh(y,f)$. If $x,y\notin\ker f$ are linearly independent, by \ref{spade.X(A)inters kerf} we can pick $z^\prime \in \mathcal{X}(\A)\cap\ker f$ such that $z^\prime, z:=y-\frac{f(y)}{f(x)}x\in \mathcal{X}(\A)\cap\ker f$ are linearly independent. It is not difficult to see the linear independence of $x,y,z'$ from the linear independence of $x,z,z'$. Using $L(x\ot f)$, $L(y\ot f)$ and $L((ax+by)\ot f)$ to act on $z'$, respectively, the linearity of $L$ implies 
\[
(h(ax+by,f)-ah(x,f)-bh(y,f))z'\in\mathbb{C} x+\mathbb{C} y.
\]
Therefore $h(ax+by,f)=ah(x,f)+bh(y,f)$. In any case, $h(x,f)$ is linear with respect to $x$. Fixing $x\in \mathcal{X}(\A)$, for arbitrary elements $f_1,f_2\in\Xb(\A)$ and $a,b\in\mathbb{C}$, if $x\in\ker f_1\cap\ker f_2$, then the linearity for $f_1,f_2$ is clear; if $x\notin\ker f_1\cap\ker f_2$, by \ref{spade.X(A)inters kerf} there exists $z\in \mathcal{X}(\A)\cap\ker f_1\cap\ker f_2$ such that $z, x$ are linearly independent. Considering the actions of $L(x\ot f_1)$, $L(x\ot f_2)$ and $L(x\ot(af_1+bf_2))$ on $z$,  respectively, the linearity of $L$ yields
\[
(h(x,af_1+bf_2)-ah(x,f_1)-bh(x,f_2))z\in\mathbb{C} x.
\]
And hence $h(x,af_1+bf_2)=ah(x,f_1)+bh(x,f_2)$. In any case, $h(x,f)$ is linear with respect to $f$. 

At last, we prove that if $\mathcal{X}=\mathcal{X}(\A)$ and $L$ is continuous, then $h(x,f)$ is continuous for $x$. Fixing $f\in\Xb(\A)$ and $x\in \mathcal{X}(\A)\sm\ker f$, for an arbitrary sequence $\{x_m\}\seq X$ satisfying $x_m\to x$, there exists $M\in\Nn$ such that for any $m\geqslant M$, $x_m\notin\ker f$ (Otherwise, it will contradict with the continuity of $f$). Then for the sequence $\{x_m\}_{m=M}^\infty$, we have 
\[
h(x_m,f)=\frac{1}{f(x_m)}f(L(x_m\ot f)x_m)\to\frac{1}{f(x)}f(L(x\ot f)x)=h(x,f).
\]
Thus $h(x,f)$ is continuous at $x$ for a fixed $f\in\Xb(\A)$. Since $\mathcal{X}(\A)=\mathcal{X}$ is a linear space, the linear functional $h_f(x):=h(x,f)$ is continuous. Namely, $h(x,f)$ is continuous for $x$.
\end{proof}

Suppose that $\A$ is a subalgebra of ${\rm Alg}\mathcal{N}$ satisfying the conditions \ref{spade.annih neq 0}--\ref{spade.X(A)=X} and that $L\colon\A\to \mathcal{B(X)}$ is a Lie $n$-derivation. For an arbitrary $f\in\Xb(\A)\sm\{0\}$, denote $\I(f)=\{y\in \mathcal{X}(\A)\mid f(y)=1\}\neq\varnothing$. For any $y\in\I(f)$, one can define a mapping $\F_{f,y}\colon \mathcal{X}(\A)\to \mathcal{X}$ as follows
$$
\F_{f,y}x:=(L(x\ot f)-h(x,f)\i)y, \ \ \forall x\in \mathcal{X}(\A).\eqno{(3.13)}
$$
By linearity of $L$ and bilinearity of $h$, we can see that $\F_{af,y}=\F_{f,ay}\;(ay\in\I(f),\,a\in\mathbb{C}\sm\{0\})$. By Lemma \ref{xxsec3.6}, the following relation 
\[
(L(x\ot f)-h(x,f)\i)z=\f_{x,f}(z)x,  \ \ \forall x\in \mathcal{X}(\A)
\]
can establish a mapping $\f_{x,f}\colon \mathcal{X}(\A)\cap\ker f\to\mathbb{C}$. Obviously, $\F_{f,y}$ and $\f_{x,f}$ are both linear, and $\f_{x,f}$ is continuous. By Lemma \ref{xxsec3.6} we know that if $L$ is continuous, then $\F_{f,y}$ is also continuous. For a given $y\in\I(f)$, it is easy to see that there exists a continuous linear extension $\f_{x,f,y}\colon \mathcal{X}(\A)\to\mathbb{C}$ of $\f_{x,f}$ satisfying $\f_{x,f,y}(y)=0$. Indeed, by Hahn-Banach theorem we can construct a continuous linear extension $\f\colon \mathcal{X}(\A)\to\mathbb{C}$ of $\f_{x,f}$, and it suffices to take $\f_{x,f,y}=\f-\f(y)f$.

\begin{lem}\label{xxsec3.7}
For any $x\in X(\A)$, $f\in\Xb(\A)$ and $y\in\I(f)$, we have
\[
L(x\ot f)=\F_{f,y}x\ot f+x\ot\f_{x,f,y}+h(x,f)\i.
\]
\end{lem}
\begin{proof}
For any $z\in \mathcal{X}(\A)=\mathcal{X}$, we have $z'=z-f(z)y\in \mathcal{X}(\A)\cap\ker f$. Thus we get
\begin{align*}
(L(x\ot f)-h(x,f)\i)z=&(L(x\ot f)-h(x,f)\i)(f(z)y+z')\\
=&f(z)\F_{f,y}x+\f_{x,f,y}(z')x\\
=&f(z)\F_{f,y}x+\f_{x,f,y}(z)x\\
=&(\F_{f,y}x\ot f+x\ot \f_{x,f,y})z.\qedhere
\end{align*}
\end{proof}

\begin{lem}\label{xxsec3.8}
$\f_{x,f}$ is not related with the choice of $x$. In this case, we write $\f_f$ for $\f_{x,f}$, and write $\f_{f,y}$ for $\f_{x,f,y}$.
\end{lem}

\begin{proof}
Let us take $f\in\Xb(\A)\sm\{0\}$ and $y\in\I(f)$. For linearly independent $x_1,x_2\in \mathcal{X}\sm\{0\}$, we arbitrarily pick $z\in\ker f$. By Lemma~\ref{xxsec3.7} it follows that
\begin{align*}
f(z)\F_{f,y}(x_1+x_2)+&\f_{x_1+x_2,f,y}(z)(x_1+x_2)\\
=&((L(x_1+x_2)\ot f)-h(x_1+x_2,f)\i)z\\
=&(L(x_1\ot f)-g(x_1,f)\i+L(x_2\ot f)-h(x_2,f)\i)z\\
=&f(z)\F_{f,y}(x_1)+\f_{x_1,f,y}(z)x_1+f(z)\F_{f,y}(x_2)+\f_{x_2,f,y}(z)x_2,
\end{align*}
from which we get (note that $\f_{x,f,y}$ is an extension of $\f_{x,f}$)
\[
\f_{x_1+x_2,f}(z)-\f_{x_1,f(z)}x_1=(\f_{x_2,f}(z)-\f_{x_1+x_2,f}(z))x_2.
\]
We therefore have $\f_{x_1,f}=\f_{x_1+x_2,f}=\f_{x_2,f}$. Since $\dim\Xb>1$, we see that $\dim \mathcal{X}>1$. For linearly dependent $x_1,x_2\in \mathcal{X}$, we can choose $x_3$ such that $x_3, x_1$ are linearly independent. In view of the previous proof, we conclude that $\f_{x_1,f}=\f_{x_3,f}=\f_{x_2,f}$.
\end{proof}

\begin{lem}\label{xxsec3.9}
For any $f_1,\,f_2\in\Xb(\A)\sm\{0\}$ and $y_j\in\I(f_j)\;(j=1,2)$, we have
\[
\F_{f_1,y_1}-\F_{f_2,y_2}\in\mathbb{C}\i.
\]
\end{lem}

\begin{proof}
If $f_1$ and $f_2$ are linearly independent, then by \cite[Proposition 1.1.1]{KadisonRingrose1}, $\ker f_1$ and $\ker f_2$ do not contain mutually. Thus there are $x_1,\,x_2\in \mathcal{X}=\mathcal{X}(\A)$ such that $f_i(x_j)=\delta_{ij}$, where $\delta_{ij}$ is the Kronecker sign. For any $x\in \mathcal{X}$, it follows that
$$
\begin{aligned}
&\F_{f_1,y_1}x\ot f_1+x\ot \f_{f_1,y_1}+\F_{f_2,y_2}x\ot f_2+x\ot\f_{f_2,y_2}\\
=&L(x\ot (f_1+f_2))-h(x,f_1+f_2)\i\\
=&\F_{f_1+f_2,y'}x\ot(f_1+f_2)+x\ot\f_{f_1+f_2,y'}, 
\end{aligned}\eqno{(3.14)}
$$
where $y'\in\I(f_1+f_2)$. Applying (3.14) to $x_1-x_2$ yields 
\[
(\F_{f_1,y_1}-\F_{f_2,y_2})x=(\f_{f_1+f_2,y'}(x_1-x_2)-\f_{f_1,y_1}(x_1-x_2)-\f_{f_2,y_2}(x_1-x_2))x.
\]
If $f_1$ and $f_2$ are linearly dependent, then we can assume $f_1=af_2$ and $a\neq 0$. It is easy to see that $ay_1\in\I(f_2)$. We now arrive at 
\[
\F_{f_1,y_1}=\F_{af_2,y_1}=\F_{f_2,ay_1}.
\]
For any $y, y'\in\I(f_2)$, then $y-y'\in \mathcal{X}(\A)\cap\ker f_2$. For any $x\in \mathcal{X}(\A)$, we by Lemma \ref{xxsec3.8} get
\[
(\F_{f_2,y}-\F_{f_2,y'})x=(L(x\ot f_2)-h(x,f_2)\i)(y-y')=\f_{f_2}(y-y') x.
\]
Let us set $y=ay_1$, $y'=y_2$. Then we obtain  $\F_{f_1,y_1}-\F_{f_2,y_2}\in\mathbb{C}\i$.
\end{proof}

\begin{lem}\label{xxsec3.10}
For any $f\in\Xb(\A)$, $y\in\I(f)$, we have $f\circ\F_{f,y}+\f_{f,y}=0$.
\end{lem}
\begin{proof}
For any $z\in\ker f$, it follows from definition (3.13) that
$$
f(\F_{f,y}z)=f(L(z\ot f)y),\;\f_{f,y}(z)=\f_{y,f,y}(z)=f(L(y\ot f)z).
$$
By invoking Lemma \ref{xxsec3.6}, we have
\begin{align*}
h(y+z,f)=&f(L(y+z)\ot f)(y+z))\\
=&f(L(y\ot f)y)+f(\F_{f,y}z)+\f_{f,y}(z)+0\\
=&h(y,f)+f(\F_{f,y}z)+\f_{f,y}(z)+0.
\end{align*}
Thus $f(\F_{f,y}z)+\f_{f,y}(z)=h(z,f)=0$. Then by $f(\F_{f,y}y)=f(L(y\ot f)y)-h(y,f)=0$, $\f_{f,y}(y)=0$ and $\ker f+\I(f)= X$, we finish the proof.
\end{proof}

We are in a position to give our main result of this subsection.

\begin{theorem}\label{xxsec3.11}
Let $\mathcal{X}$ be a Banach space with $\dim \mathcal{X}>2$ and $\N$ be a non-trivial nest on $\mathcal{X}$. Suppose that $\A$ is a subalgebra of $\alg\N$ satisfying \emph{\ref{spade.annih neq 0}--\ref{spade.Ax and A^*f}} and $\dim\Xb(\A)>1$. Then each Lie-type derivation is of standard form.  More precisely, for any Lie $n$-derivation $L\colon\A\to\mathcal{B(X)}$, there exists a derivation $D\colon\A\to\mathcal{B(X)}$ and a linear mapping $H\colon\A\to\mathbb{C}I$ vanishing on all $(n-1)$-commutators on $\A$ such that $L=D+H$. In particular, if $L$ is continuous, then $D$ and $H$ are continuous as well.
\end{theorem}

\begin{proof}
For any $x\in \mathcal{X}\sm\{0\}$, by Lemma \ref{xxsec2.1}, there exists $f\in\Xb(\A)\sm\{0\}$ such that $f(x)=0$. For a fixed $y\in\I(f)$, let us define a mapping $D(\cdot)=[\F_{f,y},\cdot\,]\colon\A\to{\rm Hom}(\mathcal{X}, \mathcal{X})$. In view of Lemma \ref{xxsec3.9}, we can say that $D$ is a linear operator (bounded and $D\colon\A\to\mathcal{B(X)}$, if $L$ is continuous) which is not related with the choice of $f,y$. Let us first consider the case of  $n>2$. For an arbitrary $A\in\A$, by Lemma \ref{xxsec3.7}, a straightforward computation shows that
$$
\begin{aligned}
&L(p_n(A,x\ot f,y\ot f,\cdots,y\ot f))\\
\notag=&L(p_{n-1}(Ax\ot f-x\ot A^*f,y\ot f,\cdots,y\ot f))\\
\notag=&L(Ax\ot f-f(Ay)x\ot f-f(Ax)y\ot f)\\
\notag=&\F_{f,y} Ax\ot f+Ax\ot\f_{f,y}+h(Ax,f)\i-f(Ay)(\F_{f,y} x\ot f+x\ot\f_{f,y})\\
\notag&-f(Ax)\big(\F_{f,y} y\ot f+y\ot\f_{f,y}+h(y,f)\i\big).
\end{aligned}\eqno{(3.15)}
$$
On the other hand, we have
$$
\begin{aligned}
&L(p_n(A,x\ot f,y\ot f,\cdots,y\ot f))\\
\notag=&p_n(L(A),x\ot f,y\ot f,\cdots,y\ot f)+p_n(A,L(x\ot f),y\ot f,\cdots,y\ot f)\\
\notag&+\sum_{j=3}^np_n(A,x\ot f,\cdots,\mathop{L(y\ot f)}_{\substack{\uparrow\\ j\text{-th}}},\cdots,y\ot f).
\end{aligned}\eqno{(3.16)}
$$
We shall calculate those terms of (3.16) in turn. Let us see the first term in (3.16). 
\begin{align*}
&p_n(L(A),x\ot f,y\ot f,\cdots, y\ot f)\\
=&p_{n-1}(L(A)x\ot f-x\ot L(A)^*f,y\ot f,\cdots,y\ot f)\\
=&L(A)x\ot f-f(L(A)y)x\ot f-f(L(A)x)y\ot f.
\end{align*}
The second term in (3.16) is 
\begin{align*}
&p_n(A,L(x\ot f),y\ot f,\cdots,y\ot f)\\
=&p_n(A,\F_{f,y} x\ot f,y\ot f,\cdots,y\ot f)+p_n(A,x\ot\f_{f,y},y\ot f,\cdots,y\ot f)\\
=&\left\{\begin{array}{l l}
A\F_{f,y} x\ot f-f(Ay)\F_{f,y} x\ot f-f(A\F_{f,y} x)y\ot f&\\
+f(\F_{f,y}x)y\ot A^*f-\f_{f,y}(Ay)x\ot f-f(Ax)y\ot\f_{f,y},&\;(n \text{ is odd})\\
& \\
A\F_{f,y} x\ot f-f(Ay)\F_{f,y} x\ot f-f(A\F_{f,y} x)y\ot f&\\
+2f(Ay)f(\F_{f,y}x)y\ot f-f(\F_{f,y}x)y\ot A^*f&\\
-\f_{f,y}(Ay)x\ot f+f(Ax)y\ot\f_{f,y},&\;(n \text{ is even})
\end{array}\right..
\end{align*}
The remainder terms in $\sum$ (for the case of $3\leq j\leq n$) of (3.16) are 
\begin{align*}
&p_n(A,x\ot f,\cdots,\mathop{L(y\ot f)}_{\substack{\uparrow\\ j\text{-th}}},\cdots,y\ot f)\\
=&p_n(A,x\ot f,\cdots,\mathop{\F_{f,y} y\ot f} _{\substack{\uparrow\\ j\text{-th}}},\cdots,y\ot f)+p_n(A,x\ot f,\cdots,\mathop{y\ot\f_{f,y}} _{\substack{\uparrow\\ j\text{-th}}},\cdots,y\ot f)\\
=&\left\{\begin{array}{c l}
n=3,\;&-f(A\F_{f,y} y)x\ot f-f(Ax)\F_{f,y} y\ot f+Ax\ot\f_{f,y}\\
&-f(Ay)x\ot\f_{f,y}-\f_{f,y}(Ax)y\ot f+\f_{f,y}(x)y\ot A^*f,\\
&\\
n>3,\;n \text{ is odd},\;&\left\{\begin{array}{c l}
j=3,\;&-f(A\F_{f,y} y)x\ot f-f(Ax)\F_{f,y} y\ot f\\
&+\f_{f,y}(x)f(Ay)y\ot f-f(Ax)y\ot \f_{f,y}\\
&-\f_{f,y}y\ot A^*f,\\
3<j<n,\;& 0,
\end{array}\right.\\
&\\
n>3,\;n \text{ is even},\;&\left\{\begin{array}{c l}
j=3,\;&-f(A\F_{f,y} y)x\ot f-f(Ax)\F_{f,y} y\ot f\\
&-\f_{f,y}(x)f(Ay)y\ot f+f(Ax)y\ot \f_{f,y}\\
&+\f_{f,y}(x)y\ot A^*f,\\
3<j<n,\;& 0,
\end{array}\right.\\
&\\
n>3,\;j=n,\;& Ax\ot\f_{f,y}-f(Ay)x\ot \f_{f,y}-f(Ax)y\ot \f_{f,y}\\
&-\f_{f,y}(Ax)y\ot f+f(Ay)\f_{f,y}(x)y\ot f.
\end{array}\right.
\end{align*}
Combining equation (3.15) with equation (3.16) and considering their actions on $x$, we obtain
\begin{align*}
h(Ax,f)x-f(Ax)h(y,f)x\in\mathbb{C} y.
\end{align*}
Thus $h(Ax,f)-f(Ax)h(y,f)=0$. In the case of $n$ is odd with $n>3$. Using  (3.15) and (3.16) to act on $y$, we arrive at
\begin{align*}
&(L(A)+A\F_{f,y}-\F_{f,y}A)x-f(L(A)y)x-f(L(A)x)y-f(A\F_{f,y}x)y
\\
&+f(\F_{f,y}x)f(Ay)y-\f_{f,y}(Ay)x+\f_{f,y}(x)f(Ay)y-\f_{f,y}(x)f(Ay)y\\
&-\f_{f,y}(Ax)y+f(Ay)\f_{f,y}(x)y=0.
\end{align*}
In the case of $n$ is even with $n>3$. Taking into account the actions of  (3.15) and (3.16) on $y$, we get
\begin{align*}
&(L(A)+A\F_{f,y}-\F_{f,y}A)x-f(L(A)y)x-f(L(A)x)y-f(A\F_{f,y}x)y\\
&+2f(Ay)f(\F_{f,y}x)y-f(\F_{f,y}x)f(Ay)y-\f_{f,y}(Ay)x-\f_{f,y}(x)f(Ay)y\\
&+\f_{f,y}(x)f(Ay)y-\f_{f,y}(Ax)y+f(Ay)\f_{f,y}(x)y=0.
\end{align*}
For the case of $n=3$, by an analogous calculation we have
\begin{align*}
&(L(A)+A\F_{f,y}-\F_{f,y}A)x-f(L(A)y)x-f(L(A)x)y-f(A\F_{f,y}x)y\\
&+f(\F_{f,y}x)f(Ay)y-\f_{f,y}(Ay)x-f(A\F_{f,y}y)x-\f_{f,y}(Ax)y\\
&+\f_{f,y}(x)f(Ay)y=0.
\end{align*} 
Therefore, for the case of $n>3$, there exist $a_{A,f,y},\;b_{A,f,y}\in\mathbb{C}$ such that
$$
\begin{aligned}
(L(A)+A\F_{f,y}-\F_{f,y}A)x+a_{A,f,y}x+b_{A,f,y}y=0.
\end{aligned}\eqno{(3.17)}
$$
Since $\dim \mathcal{X}>2$, $\dim\ker f>1$. Thus there exists $z\in\ker f$ such that $z,x$ are linearly independent. Then $x,\;y,\;y+z$ are linearly independent. Indeed, if $ax+by+c(y+z)=0$ ($a,b,c\in\mathbb{C}$), then $ax+(b+c)y+cz=0$. While $y\notin\ker f$, and hence $b+c=0$. By the fact that $x,z$ are linearly independent, it follows that $a=b=c=0$. Note that $y+z\in\I(f)$. Repeating the above discussion again gives 
$$
\begin{aligned}
(L(A)+A\F_{f,y+z}-\F_{f,y+z}A)x+a_{A,f,y+z}x+b_{A,f,y+z}(y+z)=0.
\end{aligned}\eqno{(3.18)}
$$
Note that $D(A)=[\F_{f,y},A]=[\F_{f,y+z},A]$. By (3.17) and (3.18) we conclude 
$$
(a_{A,f,y}-a_{A,f,y+z})x+b_{A,f,y}y-b_{A,f,y+z}(y+z)=0,
$$
which entails $b_{A,f,y}=0$. It follows from (3.17) that $(L(A)-D(A))x\in\ker f$. For an arbitrary $y'\in\I(f)$, applying $f$ to the following equation 
\[
(L(A)+A\F_{f,y'}-\F_{f,y'}A)x+a_{A,f,y'}x+b_{A,f,y'}y'=0,
\]
we then obtain $b_{A,f,y'}=0$. So $(L(A)-D(A))x=-a_{A,f,y'}x$. The arbitrariness of $y'$ shows that $a_{A,f,y}$ is not related with the choice of $y\in\I(f)$. Thus we denote it by $a_{A,f}$. Similarly, changing $f\in\Xb(\A)\sm\{0\}$ satisfying the condition $f(x)=0$, we can see that $a_{A,f}$ is merely related with the choice of $x$. Thus we may write $H_x(A):=-a_{A,f}$. Therefore the following equality holds:
$$
\begin{aligned}
L(A)x=D(A)x+H_x(A)x, \ \ \forall x\in \mathcal{X}.
\end{aligned}\eqno{(3.19)}
$$

For the case of $n=2$, if $A^*f\neq 0$, fixing $y'\in\I(A^*f)$, by Lemma \ref{xxsec3.9} we assert that $\F_{A^*f,y'}-\F_{f,y}=c\i$. Let us put  $\p_{A^*f}=\f_{f,y}+cf$. Furthermore, we can calculate 
\begin{align*}
L(p_2(A,x\ot f))=&L(Ax\ot f-x\ot A^*f)\\
=&\F_{f,y} Ax\ot f+Ax\ot\f_{f,y}+h(Ax,f)\i\\
&-\F_{f,y} x\ot A^*f-x\ot\p_{A^*f}-h(x,A^*f)\i
\end{align*}
and
\begin{align*}
L(p_2(A,x\ot f))=&p_2(L(A),x\ot f)+p_2(A,L(x\ot f))\\
=&L(A)x\ot f-x\ot L(A)^*f+A\F_{f,y} x\ot f\\
&+Ax\ot\f_{f,y}-\F_{f,y} x\ot A^*f-x\ot A^*\f_{f,y}.
\end{align*}
Thus we arrive at 
$$
\begin{aligned}
(L(A)+A\F_{f,y}-\F_{f,y} A)x\ot f=&x\ot L(A)^*f+x\ot A^*\f_{f,y}+h(Ax,f)\i\\
\notag&-x\ot \p_{A^*f}-h(x,A^*f)\i.
\end{aligned}\eqno{(3.20)}
$$
By an analogous manner, we also have (Let us write $y_1=y$, $y_2=y+x$)
$$
\begin{aligned}
(L(A)+A\F_{f,y}-\F_{f,y} A)y_i\ot f=&y_i\ot L(A)^*f+y_i\ot A^*\f_{f,y}+h(Ay_i,f)\i\\
\notag&-y_i\ot \p_{A^*f}-h(y_i,A^*f)\i.
\end{aligned}\eqno{(3.21)}
$$
If $A^*f=0$, adopting similar discussion we can also obtain equation (3.20) and equation (3.21) only if we substitute $0$ for $\p_{A^*f}$ and $h(\cdot ,A^*f)$. Applying (3.20) to $y$ yields
$$
\begin{aligned}
(L(A)+A\F_{f,y}-\F_{f,y} A)x=&\big(f(L(A)y)+\f_{f,y}(Ay)-\p_{A^*f}(y)\big) x\\
\notag&+\big(h(Ax,f)-h(x,A^*f)\big)y.
\end{aligned}\eqno{(3.22)}
$$
Applying (3.21) to $x$ gives
\begin{align*}
0=\big(f(L(A)x)+\f_{f,y}(Ax)-\p_{A^*f}(x)\big)y_i+\big(h(Ay_i,f)-h(y_i,A^*f)\big)x.
\end{align*}
It is not difficult to see that $x$ and $y_i(i=1,2)$ are linearly independent.  So $h(Ay_1,f)-h(y_1,A^*f)=0$ and $h(Ay_2,f)-h(y_2,A^*f)=0$. Their difference gives $h(Ax,f)-h(x,A^*f)=0$. By invoking (3.22) we know that  $f(L(A)y)+\f_{f,y}(Ay)-\p_{A^*f}(y)$ is merely related with the choice of $x$ and $A$, and we denote it by $H_x(A)$. Therefore for any $x\in \mathcal{X}$, the equality (3.19) still holds true in the case of $n=2$. That is, 
$$
\begin{aligned}
L(A)x=[\Phi_{f, y}, A]x+H_x(A)x, \ \ \forall x\in \mathcal{X}.
\end{aligned}
$$

Finally, for $x_j\in \mathcal{X}\sm\{0\}$ $(j=1,2)$ and $f_j\in\Xb(\A)\sm\{0\}$ satisfying $f_j(x_j)=0$ (by Lemma \ref{xxsec2.1} there are such $f_j$'s), since $\ker f_1\cap\ker f_2$ has codimension no more than $2$, there is $z\in(\ker f_1\cap\ker f_2)\sm\{0\}$. By what has been proved, we assert $H_{x_1}(A)=-a_{A,f_1}=H_z(A)=-a_{A,f_2}=H_{x_2}(A)$. This shows that  $H_x(A)$ is not related to the choice of $x$. One can establish a mapping $H\colon\A\to\mathbb{C}\i$ by $H(A)x:=H_x(A)x\;(\forall x\in \mathcal{X})$. In light of (3.19), we know that $H$ is linear (and continuous if $L$ is continuous). Then by (3.19) again, one can see $D=L-H$, which implies $D\colon\A\to\mathcal{B(X)}$. It follows from Proposition \ref{xxsec2.5} that $H$ vanishes on all $(n-1)$-commutators on $\A$. Thus $L=D+H$ is the desired  standard form.
\end{proof}

\begin{coro}\label{xxsec3.12}
Let $\mathcal{X}$ be a Banach space with $\dim \mathcal{X}>2$ and $\N$ be a non-trivial nest on $\mathcal{X}$. If $\mathcal{X}_-\sneq \mathcal{X}$, then each Lie-type derivation on ${\rm Alg}\N$ is of standard form. That is, for any Lie $n$-derivation $L\colon{\rm Alg}\N\to\mathcal{B(X)}$, there exists a derivation $D\colon{\rm Alg}\N\to\mathcal{B(X)}$ and a linear mapping $H\colon{\rm Alg}\N\to\mathbb{C}I$ vanishing on all $(n-1)$-commutators on ${\rm Alg}\N$ such that $L=D+H$. In particular, if $L$ is continuous, then $D$ and $H$ are continuous as well.
\end{coro}

\begin{proof}
If $\cdim \Xm>1$, then by Hahn-Banach theorem we can construct $x_1,x_2\in\Xm$ and $f_1,f_2\in\Xb$ such that $f_i(x_j)=\delta_{ij}$, where $\delta_{ij}$ is the Kronecker sign. Thus $f_1$ and $f_2$ are linearly independent, which entails $\dim\Xb>1$. Then by Theorem \ref{xxsec3.11} and Lemma \ref{xxsec2.8},  we immediately get the conclusion.

If $\cdim\Xm=1$, then for any $f\in\Xb\sm\{0\}$, we have $\cdim\ker f=1$ and $\Xm\seq\ker f$. By Lemma \ref{xxsec2.2} it follows that $\Xm=\ker f$. Taking into account \cite[Proposition~1.1.1]{KadisonRingrose1} we get $\dim\Xb=1$. In view of Theorem \ref{xxsec3.3}, the proof is completed. 
\end{proof}

Let $\mathcal{X}$ be a normed space with a nest $\N$. We denote the null element of $\mathcal{X}^*$ by $\0$. Define $\N^*:=\{N^\bot\mid N\in\N\}$. It is not difficult to verify that $\N^*$ is a nest of $\mathcal{X}^*$. We will see that the condition $\Xm\sneq \mathcal{X}$ is dual with $\{\0\}\sneq\{\0\}_+$ in the following sense.

\begin{prop}\label{xxsec3.13}
Let $\N$ be a nest on a normed space $\mathcal{X}$. Then $\Xm\sneq \mathcal{X}$ if and only if $\{\0\}\sneq\{\0\}_+$.
\end{prop}

\begin{proof}
It follows from $\Xm\sneq \mathcal{X}$ that $\{\0\}=\mathcal{X}^\bot\sneq\Xb$, in which the strict inclusion relation is given by Hahn-Banach theorem. For any $f\in \mathcal{X}^*$, we have
\begin{align*}
f\in\Xb&\Leftrightarrow\forall x\in N\sneq \mathcal{X},\,f(x)=0\Leftrightarrow\forall N\sneq \mathcal{X},\, N\seq\ker f\\
&\Leftrightarrow\forall N\sneq \mathcal{X},\,f\in N^\bot\Leftrightarrow f\in\{\0\}_+.
\end{align*}
Thus $\Xb=\{\0\}_+$, and hence $\{\0\}\sneq\{\0\}_+$.

Take $f\in\{\0\}_+\sm\{\0\}$, and then for any $x\in N\sneq \mathcal{X}$, we have $f(x)=0$. The continuity of $f$ implies that $f\Big(\bigvee_{N\sneq \mathcal{X}}N\Big)=\{0\}$. Since $f\neq 0$, we get $f(\mathcal{X})\neq\{0\}$. Therefore $\Xm\sneq \mathcal{X}$.
\end{proof}

\begin{rem}\label{xxsec3.14}
By an analogous manner, one can show that $\mathcal{X}^*_-\sneq \mathcal{X}^*$ if $\{0\}\sneq\{0\}_+$. Indeed, for an arbitrary $f\in \mathcal{X}^*_-=\big(\bigvee_{N^\bot\sneq \mathcal{X}^*}N^\bot\big)$, we know that  $f=\lim_{n\to\infty}f_n$ holds true, where $f_n\in\bigcup_{N^\bot\sneq \mathcal{X}^*}N^\bot$. It follows that $f\in\big(\bigwedge_{N\neq\{0\}}N\big)^\bot$. Thus $\mathcal{X}^*_-\seq\{0\}_+^\bot\sneq\{0\}^\bot=\mathcal{X}^*$.
\end{rem}

For a nest $\N$ on a Banach space $\mathcal{X}$, we define $(\alg\N)^*=\{A^*\mid A\in\alg\N\}$. It is easy to figure out that $(\alg\N)^*$ is a subalgebra of $\alg\N^*$. Since $*\colon\mathcal{B(X)}\to\B(\mathcal{X}^*)$ is an isometry, the completeness of $\alg\N$ entails that $(\alg\N)^*$ is complete. Thus $(\alg\N)^*$ is a closed subalgebra of $\alg\N^*$.

\begin{lem}\label{xxsec3.15}
If $L\colon\alg\N\to \mathcal{B(X)}$ is a Lie $n$-derivation, then the mapping
\begin{align*}
L^*\colon(\alg\N)^*&\to\mathcal{B(X}^*)\\
A^*&\mapsto L(A)^*
\end{align*}
is a Lie $n$-derivation on $(\alg\N)^*$. If $L$ is continuous, then $L^*$ is continuous.
\end{lem}

\begin{proof}
Linearity of $L^*$ is straightforward. If $L$ is continuous, since $*$ is continuous, the continuity of $L^*$ follows. For arbitrary elements $A_1,\cdots,A_n\in\alg\N$, we obtain
\begin{align*}
L^*(p_n(A_1^*,\cdots,A_n^*))=&(-1)^{n-1}L^*(p_n(A_1,\cdots,A_n)^*)\\
=&(-1)^{n-1}L(p_n(A_1,\cdots,A_n))^*\\
=&(-1)^{n-1}\Big(\sum_{i=1}^np_n(A_1,\cdots,L(A_i),\cdots,A_n)\Big)^*\\
=&(-1)^{n-1}\sum_{i=1}^np_n(A_1,\cdots,L(A_i),\cdots,A_n)^*\\
=&\sum_{i=1}^np_n(A_1^*,\cdots,L^*(A_i^*),\cdots,A_n^*). 
\end{align*}
This shows that $L^*$ is a Lie $n$-derivation.
\end{proof}

Note that $(x\ot f)^*=f\ot x^{**}\;(x\in \mathcal{X}, f\in \mathcal{X}^*)$, from which it is easy to see that if $\dim \mathcal{X}>2$ and $\{0\}\sneq\{0\}_+$. Then $\A:=(\alg\N)^*$, as a subalgebra of $\alg\N^*$,  satisfies \ref{spade.annih neq 0}--\ref{spade.Ax and A^*f}. Indeed, $(\mathcal{X}^*)^\bot_-(\A)=\{x^{**}\mid x\in \{0\}_+\}$ is a linear space, and hence $\A$ satisfies \ref{spade.annih neq 0}. Applying Remark \ref{xxsec3.14} yields \ref{spade.y ot f in A}. Considering Lemma \ref{xxsec2.3}, one can see that \ref{spade.X(A)=X} holds true (For any $f\in \mathcal{X}^*$, pick $x\in\{0\}_+\sm\{0\}$ and we get $(x\ot f)^*\in\A$); other conditions are easily to be verified.

\begin{coro}\label{xxsec3.16}
Let $\mathcal{X}$ be a Banach space with $\dim \mathcal{X}>2$ and $\N$ be a non-trivial nest on $\mathcal{X}$. If $\{0\}\sneq\{0\}_+$, then each Lie-type derivation on ${\rm Alg}\N$ is of standard form. More precisely, for any Lie $n$-derivation $L\colon{\rm Alg}\N\to\mathcal{B(X)}$, there exist a derivation $D\colon{\rm Alg}\N\to\mathcal{B(X)}$ and a linear mapping $H\colon{\rm Alg}\N\to\mathbb{C}I$ vanishing on all $(n-1)$-commutators on ${\rm Alg}\N$ such that $L=D+H$. In particular, if $L$ is continuous, then $D$ and $H$ are continuous as well.
\end{coro}

\begin{proof}
Let us write $\A=(\alg\N)^*$. As is illustrated in the above, we know that  $\dim(\mathcal{X}^*)^\bot_-(\A)\geqslant 1$. It follows from Theorem \ref{xxsec3.3} and Theorem \ref{xxsec3.11} that the Lie $n$-derivation $L^*$ has the standard form, i.e., there is a derivation $D'\colon\A\to\B(\mathcal{X}^*)$ and linear mapping $H'\colon\A\to\mathbb{C}\i^*$ vanishing on all $(n-1)$-th commutators on $(\alg\N)^*$ such that $L^*=D'+H^\prime$ (by Lemma \ref{xxsec3.15}, $D^\prime$ and $H^\prime$ are continuous when $L$ is continuous). Let us define  $H(A)^*:=H'(A^*)$. It is easy to check that $H\colon\alg\N\to\mathbb{C}\i$ is a linear mapping (is also continuous when $L$ is continuous). For any $A,B\in\alg\N$, we have
\begin{align*}
\big((L-H)(AB)\big)^*=&(L^*-H')(B^*A^*)\\
=&D'(B^*A^*)=D'(B^*)A^*+B^*D'(A^*)\\
=&\big((L-H)(B)\big)^*A^*+B^*\big((L-H)(A)\big)^*\\
=&\big((L-H)(A)B+A(L-H)(B)\big)^*.
\end{align*}
Since $*$ is isometric, it follows that
$$
(L-H)(AB)=(L-H)(A)B+A(L-H)(B). 
$$
Thus $D=L-H$ is a derivation (is also continuous if $L$ is continuous). 
\end{proof}

\subsection{The case of $\dim\Xb=0$}
\ \\

By Theorem \ref{xxsec3.3}, Corollary \ref{xxsec3.5}, Theorem \ref{xxsec3.11}, Corollary \ref{xxsec3.12} and Corollary \ref{xxsec3.16}, it suffices to consider the case of
\begin{enumerate}[fullwidth,itemindent=9em,label=$(\clubsuit)$]
\item $\dim \mathcal{X}>2$, $\Xm=\mathcal{X}$ and $\{0\}=\{0\}_+$. \label{clubsuit.dim X>2, X-=X,etc.}
\end{enumerate}
Then there is a net $\{N_\l\}_{\l\in\L}\seq\N$ satisfying $\bigvee_{\l\in\L}N_\l=\mathcal{X}$ (e.g., take $\L=\N$, $N_\l=\l\in\N$).

\begin{lem}\label{xxsec3.17}
Let $\mathcal{X}$ be a normed space and $\N$ be a non-trivial nest of $\mathcal{X}$ satisfying \ref{clubsuit.dim X>2, X-=X,etc.}. Then the following statements are true: 
\begin{enumerate}
\item[{\rm (1)}] For any $N\in\N\sm\{\{0\}, \mathcal{X}\}$, we have $\dim N=\infty$, $\cdim N=\infty$.
\item[{\rm (2)}] For any $N\in\N\sm\{\{0\}, \mathcal{X}\}$ and $f\in\ \mathcal{X}^*$, we have $\dim N\cap\ker f=\infty$.
\item[{\rm (3)}] For any net $\{N_\l\}_{\l\in\L}\seq\N$ satisfying $\bigvee_{\l\in\L}N_\l=\mathcal{X}$ and $f\in \mathcal{X}^*\sm\{0\}$, there exists $\a\in\L$ such that $N_\a\not\seq\ker f$.
\end{enumerate}
\end{lem}

\begin{proof}
(1) We use reduction to absurdity. Suppose that $N_0\in\N\sm\{\{0\}, \mathcal{X}\}$ is of finite dimension. By $\{0\}=\{0\}_+=\bigwedge\{N\in\N\mid \{0\}\sneq N\}$ and the total-orderedness of $\N$ it follows that there exists $m\in\Nn$ and $N_1,\cdots,N_m\in\N\sm\{\{0\}\}$ with $N_j\seq N_0$, such that $\{0\}=\bigwedge_{j=0}^n N_j$, while $\dim\bigwedge_{j=0}^n N_j\geqslant 1$, which is a contradiction.

Suppose that $N_0\in\N\sm\{\{0\}, \mathcal{X}\}$ is of finite codimension. Then the set $\{N\in\N\mid N_0\sneq N\sneq \mathcal{X}\}$ is a finite set. Indeed, for $N_0\sneq N$, by the fact $\mathcal{X}/N\cong(\mathcal{X}/N_0)/(N/N_0)$, $\dim \mathcal{X}/N$ has at most $\dim \mathcal{X}/N_0=\cdim N_0<\infty$ possible choices. Moreover, since $\N$ is totally ordered, we see that $\mathcal{X}/N$ has at most $\cdim N_0$ possible choices. Let us write $\{N\in\N\mid N_0\sneq N\sneq \mathcal{X}\}=\{N_j\}_{j=0}^m$, which is an increasing sequence under inclusion relation. Then
$$
\mathcal{X}_-=\ol{\lspan\bigcup\{N\in\N\mid N\sneq \mathcal{X}\}}=\ol{\lspan\bigcup_{j=0}^mN_j}=N_m\sneq \mathcal{X},
$$
which contradicts with $\mathcal{X}_-=\mathcal{X}$.

(2) If $N\seq\ker f$, then by (1) we have $\dim N\cap\ker f=\dim N=\infty$. If $N\not\seq\ker f$, then there exists $y\in N$ such that $f(y)=1$. Thus for any  $x\in N$, $x-f(x)y\in\ker f$, namely $(\i-y\ot f)\mathcal{X}\seq \ker f$. Therefore $(\i-y\ot f)N\seq N\cap\ker f$, where $(\i-y\ot f)N$ is an infinite-dimensional subspace. Henceforth we assert $\dim N\cap\ker f=\infty$.

(3) We use reduction to absurdity. If $N_\l\seq\ker f$ holds true for any  $\l\in\L$, then $\bigvee_{\l\in\L}N_\l=\mathcal{X}$. This implies that $\ker f=\mathcal{X}$, a contradiction.
\end{proof}

Without loss of generality, we next assume that $\{N_\l\}_{\l\in\L}\seq\N\sm\{\{0\}, \mathcal{X}\}$ is a net satisfying $\bigvee_{\l\in\L}N_\l=\mathcal{X}$. For $f\in \mathcal{X}^*\sm\{0\}$, put $\I(f):=\{y\in \mathcal{X}\mid f(y)=1\}\neq\varnothing$.

\begin{lem}\label{xxsec3.18}
Let $\mathcal{X}$ be a Banach space and $\N$ be a non-trivial nest satisfying \ref{clubsuit.dim X>2, X-=X,etc.}. If $L\colon\alg\N\to\mathcal{B(X)}$ is a Lie $n$-derivation, then for any $\l\in\L$, there exists a bilinear functional $h_\l\colon N_\l\times (N_\l)_-^\bot\to\mathbb{C}$ such that 
$$
(L(x\ot f)-h_\l(x,f)\i)(N_\l\cap\ker f)\seq\mathbb{C} x,  \ \ \forall x\in N_\l, \ \ \forall f\in (N_\l)_-^\bot.
$$
In particular, if $L$ is continuous, then $h_\l(x,f)$ is continuous for $x$.
\end{lem}

\begin{proof}
For any $x\in N_\l$ and $f\in (N_\l)_-^\bot$, by Lemma \ref{xxsec3.17} we know that $\dim (N_\l)_-^\bot=\infty$. Then there exists $f_1\in(N_\l)_-^\bot$ such that $f_1, f$ are linearly independent. By \cite[Proposition 1.1.1]{KadisonRingrose1}, there exists $y_1\in \mathcal{X}$ such that $f(y_1)=0$ and $f_1(y_1)=1$. For any $A\in\alg\N$ and $z\in N_\l\cap\ker f$, if $n>2$, it follows that
\begin{align*}
0=&L(p_n(x\ot f,z\ot f_1,A,\cdots,A))\\
=&p_n(L(x\ot f),z\ot f_1,A,\cdots,A)+p_n(x\ot f,L(z\ot f_1),A,\cdots,A)\\
=&p_{n-1}([L(x\ot f),z\ot f_1]+[x\ot f,L(z\ot f_1)],A\,\cdots,A).
\end{align*}
By invoking Lemma \ref{xxsec2.6} and Lemma \ref{xxsec2.8} we conclude
$$
[L(x\ot f),z\ot f_1]+[x\ot f,L(z\ot f_1)]=c\i,
$$
where $c\in\mathbb{C}$. Since $\i$ is of infinite rank, we have $c=0$. Thus
$$
[L(x\ot f),z\ot f_1]+[x\ot f,L(z\ot f_1)]=0.\eqno{(3.23)}
$$
If $n=2$, (3.23) can be achieved by calculating $0=L(p_2(x\ot f,z\ot f_1))$. Considering the action of (3.23) on $y_1$, we obtain
\[
L(x\ot f)z=f_1(L(x\ot f)y_1)z-f(L(z\ot f_1)y_1)x.
\]
We next illustrate that the equality $f_1(L(x\ot f)y_1)=f_2(L(x\ot f)y_2)$ holds true, where $f_j\in (N_\l)_-^\bot$ is linearly independent with $f$ and $y_j\in\I(f)\cap\ker f_j$ $(j=1,2)$. By what has been proved, we get
$$
L(x\ot f)z=f_j(L(x\ot f)y_j)z-f(L(z\ot f_j)y_j)x\quad(j=1,2). \eqno{(3.24)}
$$
Since $N_\l\cap\ker f$ is of infinite dimension, we can take $z'\in N_\l\cap\ker f$ such that $z', x$ are linearly independent. Similarly, we have
$$
L(x\ot f)z'=f_j(L(x\ot f)y_j)z'-f(L(z'\ot f_j)y_j)x\quad(j=1,2).\eqno{(3.25)}
$$
Then by linear independence, the above two equalities (3.24) and (3.25)  entail $f_1(L(x\ot f)y_1)=f_2(L(x\ot f)y_2)$. Let us define $h_\l(x,f)=f_1(L(x\ot f)y_1)$, which is well-defined by the above deduction. It is clear that $h_\l(x,f)$ is continuous for $x\in N_\l$ if $L$ is continuous.

Let us show the bilinearity of $h_\l$. In view of our above construction, linearity of $h_\l(x,f)$ regarding to $x$ is straightforward. For any $f_1,f_2\in (N_\l)_-^\bot$, if $f_1$ and $f_2$ are linearly dependent, the case is trivial; if $f_1$ and $f_2$ are linearly independent, from $\dim (N_\l)_-^\bot=\infty$ it follows that there exists $f_{12}\in (N_\l)_-^\bot$ such that $f_{12},f_1,f_2$ are linearly independent. By invoking  \cite[Proposition 1.1.1]{KadisonRingrose1}, we have that $\ker f_1\cap\ker f_2\not\seq\ker f_{12}$. Then there exists $y_{12}\in\I(f_{12})\cap\ker f_1\cap\ker f_2$, and hence
\begin{align*}
h_\l(x,f_1+f_2)=&f_{12}(L(x\ot(f_1+f_2))y_{12})\\
=&f_{12}(L(x\ot f_1)y_{12})+f_{12}(L(x\ot f)y_{12})\\
=&h_\l(x,f_1)+h_\l(x,f_2).
\end{align*}
That is, $h_\l(x,f)$ is linear regarding to $f$ as well.\qedhere
\end{proof}

For an arbitrary $f\in (N_\l)_-^\bot\sm\{0\}$ and $y\in\I(f)$, with similar discussion as in Section \ref{sect.dim Xb>1}, one can define a linear operator (bounded if $L$ is continuous) $\F_{\l,f,y}\colon N_\l\to \mathcal{X}$ by
\[
\F_{\l,f,y}x:=(L(x\ot f)-h_\l(x,f)\i)y, \ \ \forall x\in N_\l.
\]
For any $z\in N_\l\cap\ker f$, the following equation
\[
(L(x\ot f)-h_\l(x,f))z=\f_{\l,x,f}(z)x
\]
establishes a bounded linear functional $\f_{\l,x,f}\colon N_\l\cap\ker f\to\mathbb{C}$. Adopting the same strategy in Section \ref{sect.dim Xb>1}, we can find a continuous extension $\f_{\l,x,f,y}\colon \mathcal{X}\to\mathbb{C}$ of $\f_{\l,x,f}$ such that $\f_{\l,x,f,y}(y)=0$.

\begin{lem}\label{xxsec3.19}
Let $\mathcal{X}$ be a Banach space and $\N$ is a non-trivial nest on $\mathcal{X}$ satisfying the condition \ref{clubsuit.dim X>2, X-=X,etc.}.
\begin{enumerate}
\item[{\rm (1)}] For any $\l\in\L$, $x\in N_\l$ and $f\in (N_\l)_-^\bot$, we have 
$$
L(x\ot f)=\F_{\l,f,y}x\ot f+x\ot \f_{\l,x,f,y}+h_\l(x,f)\i, \ \ \forall y\in\I(f). 
$$

\item[{\rm (2)}] $\f_{\l,x,f}$ is not related with the choice of $x\in N_\l$, and thus we write $\f_\l$ for $\f_{\l,x}$, and write $\f_{\l,f,y}$ for $\f_{\l,x,f,y}$.

\item[{\rm (3)}] For any $f_j\in (N_\l)_-^\bot\sm\{0\}$ and any $y_j\in\I(f_j)$ $(j=1,2)$, we have $\F_{\l,f_1,y_1}-\F_{\l,f_2,y_2}\in\mathbb{C}\i|_{N_\l}$.
\end{enumerate}
\end{lem}

The proof of this lemma is parallel with proofs of Lemma \ref{xxsec3.7}, Lemma \ref{xxsec3.8} and Lemma \ref{xxsec3.9}, and hence we do not repeat them here.

\begin{lem}\label{xxsec3.20}
Let $\mathcal{X}$ be a Banach space and $\N$ be a non-trivial nest on $\mathcal{X}$ satisfying the condition \ref{clubsuit.dim X>2, X-=X,etc.}. If $L\colon\alg\N\to\mathcal{B(X)}$ is a Lie $n$-derivation, then for any $A\in\alg\N$ and $\l\in\L$, the equality $L(A)|_{N_\l}=D_\l(A)+H_\l(A)$ holds true, where $D_\l\colon\alg\N\to\B(N_\l, \mathcal{X})$ is a derivation given by $D_\l(A)x=\F_{\l,f,y}Ax-A\F_{\l,f,y}x\,(\forall f\in N_\l^\bot\sm\{0\}, \forall y\in\I(f), \forall x\in N_\l)$ and $H_\l\colon\alg\N\to\mathbb{C}\i|_{N_\l}$ is a linear mapping. Moreover, if $L$ is continuous, then $D_\l$ and $H_\l$ are continuous.
\end{lem}

\begin{rem}\label{xxsec3.21}
For any $B\in\B(N_\l, \mathcal{X})$ (or ${\rm Hom}(N_\l, \mathcal{X})$) and $A\in\alg\N$, let us define $BA\in\B(N_\l, \mathcal{X})$ (or ${\rm Hom}(N_\l, \mathcal{X})$) by $BAx=B(Ax)\,(\forall x\in N_\l)$. Then $\B(N_\l, \mathcal{X})$ (or ${\rm Hom}(N_\l, \mathcal{X})$) becomes an $(\alg\N,\alg\N)$-bimodule. For any $f_j\in N^\bot_\l\sm\{0\}$ and $y_j\in\I(F_j)$, we by Lemma \ref{xxsec3.19} assert that $[\F_{\l,f_1,y_1},A]=[\F_{\l,f_2,y_2},A]$ for all $A\in\alg\N$. Thus $D_\l\colon\alg\N\to{\rm Hom}(N_\l, \mathcal{X})$ is well-defined. In what follows we will prove the fact that $D_\l\colon\alg\N\to\B(N_\l, \mathcal{X})$.
\end{rem}

\begin{proof}
The case of $n=3$ is given in \cite{SunMa}, and hence we only consider the cases of $n>3$ and $n=2$, respectively. Let us first deal with the case of $n>3$. For an arbitrary $f\in N_\l^\bot\sm\{0\}$, it follows from Lemma \ref{xxsec3.17} that there exists $\a\in\L$ such that $N_\a\not\seq\ker f$. And hence there exists $y\in\I(f)\cap N_\a$. The fact that $\N$ is totally ordered implies that $N_\l\seq N_\a$. Thus for an arbitrary $g\in N_\a^\bot\sm\{0\}$, $f$ and $g$ are linearly independent. By \cite[Proposition 1.1.1]{KadisonRingrose1} we know that there exists $z\in\I(g)\cap\ker f$.

Given $f'\in N_\l^\bot\sm\{0\}$ and $y'\in\I(f)$, by invoking Lemma \ref{xxsec3.19} we see that $\F_{\l,f,y}-\F_{\l,f',y'}=c_{\l,f',y'}\i|_{N_\l}$, and we define $\p_{\l,f'}:=\f_{\l,f,y}+c_{\l,f',y'}\i$, where $c_{\l,f',y'}=0$ if $f'=f$, $y'=y$. Then by Lemma \ref{xxsec3.19} we get
$$
L(x\ot f')=\F_{\l,f,y} x\ot f'+x\ot \p_{\l,f'}+h_\l(x,f')\i.\eqno{(3.26)}
$$

Under the above notations, for arbitrary $x\in N_\l$, a direct computation shows that
$$
\begin{aligned}
&L(p_n(A,x\ot f,y\ot f,\cdots,y\ot f,y\ot g))\\
\notag=&L(p_{n-1}(Ax\ot f-x\ot A^*f,y\ot f,\cdots,y\ot f,y\ot g))\\
\notag=&L(Ax\ot g-f(Ay)x\ot g)\\
\notag=&\F_{\l,f,y} Ax\ot g+Ax\ot\p_{\l,g}+h_\l(Ax,g)\i-f(Ay)\F_{\l,f,y} x\ot g\\
\notag&-f(Ay)x\ot \p_{\l,g}-f(Ay)h_\l(x,g)\i.
\end{aligned}\eqno{(3.27)}
$$
On the other hand, we have 
$$
\begin{aligned}
&L(p_n(A,x\ot f,y\ot f,\cdots,y\ot f,y\ot g))\\
\notag=&p_n(L(A),x\ot f,y\ot f,\cdots,y\ot f,y\ot g)\\
\notag&+p_n(A,L(x\ot f),y\ot f,\cdots,y\ot f,y\ot g)\\
\notag&+\sum_{j=3}^{n-1}p_n(A, x\ot f,\cdots,\mathop{L(y_j\ot f)}_{\substack{\uparrow\\ j\text{-th}}},\cdots,y\ot g)\\
\notag&+p_n(A,x\ot f,y\ot f,\cdots,y\ot f,L(y\ot g)).
\end{aligned}\eqno{(3.28)}
$$
Let us calculate each term in the right side of (3.28). The first and second terms in the right side of (3.28) are
\begin{align*}
&p_n(L(A),x\ot f,y\ot f,\cdots,y\ot f,y\ot g)\\
=&L(A)x\ot g-f(L(A)y)x\ot g-f(L(A)x)y\ot g-g(L(A)x)y\ot f
\end{align*}
and
\begin{align*}
&p_n(A,L(x\ot f),y\ot f,\cdots,y\ot f,y\ot g)\\
=&A\F_{\l,f,y} x\ot g-f(Ay)\F_{\l,f,y} x\ot g-f(A\F_{\l,f,y} x)y\ot g+f(\F_{\l,f,y} x)f(Ay)y\ot g\\
&-g(A\F_{\l,f,y} x)y\ot f+f(Ay)g(\F_{\l,f,y} x)y\ot f-\f_{\l,f,y}(Ay)x\ot g.
\end{align*}
Those terms in $\sum$ for the case of $3\leqslant j\leqslant n-1$ are
\begin{align*}
&p_n(A,x\ot f,\cdots,\mathop{L(y\ot f)}_{\substack{\uparrow\\ j\text{-th}}},\cdots,y\ot g)\\
=&\left\{
\begin{array}{c l}
j=n-1,&f(\F_{\l,f,y} y)Ax\ot g-f(\F_{\l,f,y} y)f(Ay)x\ot g-\f_{\l,f,y}(Ax)y\ot g\\
&+\f_{\l,f,y}(x)f(Ay)y\ot g\\
 & \\
3<j<n-1,&f(\F_{\l,f,y} y)Ax\ot g-f(\F_{\l,f,y} y)f(Ay)x\ot g\\
 & \\
j=3,&\left\{
\begin{array}{c l}
n=4,&f(\F_{\l,f,y} y)Ax\ot g-f(A\F_{\l,f,y} y)x\ot g\\
&-\f_{\l,f,y}(Ax)y\ot g+\f_{\l,f,y}(x)f(Ay)y\ot g\\
 & \\
n>4,&\left\{
\begin{array}{c l}
n\text{ is odd},&f(\F_{\l,f,y} y)Ax\ot g\\
&-f(A\F_{\l,f,y} y)x\ot g\\
& \\
n\text{ is even},&2\f_{\l,f,y}(x)f(Ay)y\ot g\\
&+f(\F_{\l,f,y} y)Ax\ot g\\
&-f(A\F_{\l,f,y} y)x\ot g
\end{array}\right.
\end{array}\right.
\end{array}\right.
\end{align*}
The last term is
\begin{align*}
&p_n(A,x\ot f,y\ot f,\cdots,y\ot f,L(y\ot g))\\
=&p_n(A,x\ot f,y\ot f,\cdots,y\ot f,\F_{\l,f,y} y\ot g+y\ot \p_{\lg})\\
=&p_2(Ax\ot f-f(Ay)x\ot f-f(Ax)y\ot f,\F_{\l,f,y} y\ot g+y\ot \p_{\l,g})\\
=&f(\F_{\l,f,y} y)Ax\ot g-f(\F_{\l,f,y} y)f(Ay)x\ot g-f(Ax)f(\F_{\l,f,y} y)y\ot g+Ax\ot \p_{\l,g}\\
&-f(Ay)x\ot \p_{\l,g}-\p_{\l,g}(Ax)y\ot f+f(Ay)\p_{\l,g}(x)y\ot f.
\end{align*}

Comparing (3.27) with (3.28), noticing that $\i$ has infinitely rank, we have that $h_\l(Ax,g)-f(Ay)h_\l(x,g)=0$. Combining (3.27) with (3.28) and considering their actions to $z$, we obtain
$$
\begin{aligned}
0=&L(A)x+A\F_{\l,f,y} x-\F_{\l,f,y} Ax-f(L(A)y)x-f(L(A)x)y\\
\notag&-f(A\F_{\l,f,y} x)y+f(\F_{\l,f,y} x)f(Ay)y+f(\F_{\l,f,y} y)Ax-f(\F_{\l,f,y} y)f(Ay)x\\
\notag&-f(Ax)f(\F_{\l,f,y} y)y+(n-4)(f(\F_{\l,f,y} y)Ax-f(\F_{\l,f,y} y)f(Ay)x)\\
\notag&+f(\F_{\l,f,y} y)Ax-f(\F_{\l,f,y} y)f(Ay)x+\f_{\l,f,y}(x)f(Ay)y-\f_{\l,f,y}(Ax)y\\
\notag&+\chi_{(4,\infty)}(n)\cdot\left\{
\begin{array}{c l}
n\text{ is odd},&f(\F_{\l,f,y} y)Ax-f(A\F_{\l,f,y} y)x\\
n\text{ is even},&2\f_{\l,f,y}(x)f(Ay)y+f(\F_{\l,f,y} y)Ax-f(A\F_{\l,f,y} y)x
\end{array},\right.
\end{aligned}\eqno{(3.29)}
$$
where $\chi_{(4,\infty)}(n)=\left\{
\begin{smallmatrix}
1,&\text{if }n>4\\
0,&\text{if }n\leqslant 4
\end{smallmatrix}\right.$. Note that the subspace $V_x:=\lspan\{(L(A)+A\F_{\l,f,y}-\F_{\l,f,y}A)x,Ax,x\}=\lspan\{(L(A)-D_\l(A))x,Ax,x\}$ has nothing relation with the choice of $f, y$. Its dimension is no larger than $3$. Since $y+N_\a\cap\ker f$ is a coset of a infinite-dimensional subspace, $y+N_\a\cap\ker f\not\seq V_x$. Let us pick $y_x\in(y+N_\a\cap\ker f)\sm V_x$. And then $y_x\in\I(f)$. Henceforth (3.29) holds true for $y_x$ and $\F_{\l,f,y_x}$. The fact that $y_x\notin V_x$ implies that the coefficient of $y_x$ in equality (3.29) is $0$, when we substitute $y_x$ for $y$. We therefore have
\[
(L(A)-[\F_{\l,f,y},A])x=(L(A)-[\F_{\l,f,y_x},A])x\in\lspan\{Ax,x\}\seq N_\l.
\]
It follows from $y\notin N_\l$ that the coefficient of $y$ in (3.29) is $0$ (otherwise $y\in\lspan\{(L(A)-[\F_{\l,f,y},A])x,Ax,x\}$ $\seq N_\l$, a contradiction). Observing the coefficients of $Ax$ and $x$ in (3.29), we can find that
\[
(L(A)-D_\l(A))x=(L(A)-[\F_{\l,f,y},A])x=c_\l Ax+H_\l(A)x,
\]
where $c_\l\in\mathbb{C}$ is merely related with $\l,f,y$ and is not related with the choice of $x$, and $H_\l(A)\in\mathbb{C}\i$ is only related with $\l,A,f,y$ and is nothing relation with the choice of $x$.

Let us next show that $c_\l=0$. We use reduction to absurdity. Suppose that $c_\l\neq 0$. For any $A_1,\cdots,A_n\in\alg\N$ and an arbitrary $x\in N_\l$, by what has been proved we have
\begin{align*}
&D_\l(p_n(A_1,\cdots,A_n))x+c_\l p_n(A_1,\cdots,A_n)x+H_\l(p_n(A_1,\cdots,A_n))x\\
=&L(p_n(A_1,\cdots,A_n))x\\
=&\sum_{j=1}^np_n(A_1,\cdots,D_\l(A_j)+c_\l A_j,\cdots,A_n)x\\
=&D_\l(p_n(A_1,\cdots,A_n))x+n\cdot c_\l p_n(A_1,\cdots,A_n)x.
\end{align*}
Here, the last equality is due to the fact that $D_\l$ is a derivation and hence a Lie $n$-derivation (Proposition \ref{xxsec2.4}). this shows that
\[
(n-1)\cdot c_\l p_n(A_1,\cdots,A_n)x=H_\l(p_n(A_1,\cdots,A_n))x.
\]
Restricting the above equation to $N_\l$, by Lemma \ref{xxsec2.6} we obtain that 
\[
\frac{1}{(n-1)c_\l}H_\l(p_n(A_1,\cdots,A_n))=0.
\]
That is, $p_n(A_1,\cdots,A_n)|_{N_\l}=0$, which is obviously absurd. Indeed, for an arbitrary $N\in\N\sm\{\{0\}\}$ satisfying $N\sneq N_\l$ and $v\in N_\l\sm N$, we define a bounded linear functional $\wt w_0$ on the subspace $\lspan\{v+N\}$ of $\mathcal{X}/N$ by $\wt w_0(av+N):=a$. By Hahn-Banach theorem, we get an extension $\wt w\in(\mathcal{X}/N)^*$. Let $w=\wt w\circ \tau$, where $\tau\colon \mathcal{X}\to \mathcal{X}/N$ is the quotient mapping. Then $w\in N^\bot$ and $w(v)=1$. For any $u\in N\sm\{0\}$, we have $u\ot w=p_n(u\ot w,v\ot w,\cdots, v\ot w)=0$, while $u\ot w(v)=u\neq 0$, a contradiction.

Let us now consider the case of $n=2$. For any $f\in N_\l^\bot\sm\{0\}$, $y\in\I(f)$ and an arbitrary $x\in N_\l$, we get
$$
\begin{aligned}
L(p_2(A,x\ot f))=&L(Ax\ot f-x\ot A^*f)\\
=&\F_{\l,f,y}Ax\ot f+Ax\ot \f_{\l,f,y}+h_\l(Ax,f)\i\\
&-\F_{\l,f,y}x\ot A^*f-x\ot\p_{\l,A^*f}-h_\l(x,A^*f)\i
\end{aligned}\eqno{(3.30)}
$$
and
$$
\begin{aligned}
L(p_2(A,x\ot f))=&p_2(L(A),x\ot f)+p_2(A,L(x\ot f))\\
=&L(A)x\ot f-x\ot L(A)^*f+A\F_{\l,f,y}x\ot f+Ax\ot \f_{\l,f,y}\\
&-\F_{\l,f,y}x\ot A^*f-x\ot A^*\f_{\l,f,y}.
\end{aligned}\eqno{(3.31)}
$$
Combining (3.30) with (3.31) gives
$$
\begin{aligned}
(L(A)+A\F_{\l,f,y}-\F_{\l,f,y}A)x\ot f=&x\ot L(A)^*f+x\ot A^*\f_{\l,f,y}\\
\notag&+h_\l(Ax,f)\i-x\ot \p_{\l,A^*f}-h_\l(x,A^*f)\i.
\end{aligned}\eqno{(3.32)}
$$
Since $\i$ is of infinite rank, we arrive at
\[
h_\l(Ax,f)-h_\l(x,A^*f)=0.
\]
Considering the action of (3.32) on $y$, we finish the proof of this case.

Finally, the equality $L(\cdot)|_{N_\l}=D_\l(\cdot)+H_\l(\cdot)$ entails that $D_\l(\cdot)=L(\cdot)|_{N_\l}-H_\l(\cdot)\colon\alg\N\to\B(N_\l,\mathcal{X})$. If $L$ is continuous, then $\F_{\l,f,y}$ is continuous and hence $D_\l$ is bounded. By the relation $L(\cdot)|_{N_\l}=D_\l(\cdot)+H_\l(\cdot)$, we get the continuity of $H_\l$.
\end{proof}

\begin{theorem}\label{xxsec3.22}
Let $\mathcal{X}$ be a Banach space with $\dim\Xb=0$ and $\N$ be a non-trivial nest on $\mathcal{X}$. Then each Lie-type derivation on ${\rm Alg}\N$ is of standard form. More precisely, for any Lie $n$-derivation $L\colon{\rm Alg}\N\to\mathcal{B(X)}$, there exists a derivation $D\colon{\rm Alg}\N\to\mathcal{B(X)}$ and a linear mapping $H\colon{\rm Alg}\N\to\mathbb{C}I$ vanishing on all $(n-1)$-commutators on ${\rm Alg}\N$ such that $L=D+H$. In particular, if $L$ is continuous, then $D$ and $H$ are continuous as well.
\end{theorem}

\begin{proof}
As we mention at the beginning of this subsection, it suffices to consider the case of $\dim X>2$, $\Xm=\mathcal{X}$ and $\{0\}=\{0\}_+$. For arbitrary $\a,\b\in\L$, suppose that $N_\a\seq N_\b$. For any $x\in N_\a$, fixing $f\in N_\b^\bot\sm\{0\}$ and $y\in\I(f)$, then by relation (3.26) we know that the equality 
$$
(\F_{\a,f_1,y_1}-\F_{\b,f_2,y_2})x\ot f=x\ot(\p_{\b,f_2}-\p_{\a,f_1})+(h_\b(x,f)-h_\a(x,f))\i
\eqno{(3.33)}
$$
holds true for any $f_1\in N_\a^\bot$, $f_2\in N_\b^\bot$ and $y_j\in\I(f_j)$. Note that $\i$ has infinitely rank, and thus $h_\b(x,f)-h_\a(x,f)=0$. Considering the action of (3.33) on $y$, we obtain $\F_{\a,f_1,y_1}-\F_{\b,f_2,y_2}|_{N_\a}\in\mathbb{C}\i|_{N_\a}$. By Lemma \ref{xxsec3.20} we conclude that
\[
0=[\F_{\a,f_1,y_1}-\F_{\b,f_2,y_2}|_{N_\a},A]x=(H_\b(A)-H_\a(A))x,\ \ \forall A\in\alg\N. 
\]
And hence $H_\a(A)=H_\b(A)|_{N_\a}$. For each $\l\in\L$, let us denote $\wt H_\l(A)\i|_{N_\l}=H_\l(A)$. By Lemma \ref{xxsec3.20} it follows that $\wt H_\l\colon\alg\N\to\mathbb{C}I$ is a (continuous, if $L$ is continuous) linear functional, and $\wt H_\a(A)=\wt H_\b(A)$. Define $H(A):=\wt H_\l(A)\i$ for some $\l\in\L$ and write $D:=L-H$. Thus $H\colon\alg\N\to\mathbb{C}\i$ is a (bounded, if $L$ is continuous) linear operator. For any $A,B\in\alg\N$ and $x\in \bigcup_{\l\in\L}N_\l$, there is a $\l\in\L$ such that $x\in N_\l$. By Lemma \ref{xxsec3.20} it follows that
\begin{align*}
D(AB)x=&(L-H)(AB)x\\
=&D_\l(AB)x\\
=&D_\l(A)Bx+AD_\l(B)x\\
=&(L(A)-\wt H_\l(A)\i|_{N_\l})Bx+A(L(B)-\wt H_\l(B)\i|_{N_\l})x\\
=&(L(A)-\wt H_\l(A)\i)Bx+A(L(B)-\wt H_\l(B)\i)x\\
=&\big(D(A)B+AD(B)\big)x.
\end{align*}
Taking into account the fact  $\mathcal{X}=\ol{\lspan\bigcup_{\l\in\L}N_\l}=\ol{\bigcup_{\l\in\L}N_\l}$ (Since $\N$ is totally ordered, $\bigcup_{\l\in\L}N_\l$ is a linear space), we know that $D\colon\alg\N\to\mathcal{B(X)}$ is a (continuous, if $L$ is continuous) derivation. By invoking Proposition \ref{xxsec2.4} we see that  $H$ vanishes on all $(n-1)$-th commutators on $\alg\N$.
\end{proof}

\begin{rem}\label{xxsec3.23}
We would like to point out that the main results in \cite{Lu, SunMa, ZhangWuCao} are direct consequences of our main Theorem \ref{xxsec3.1}.
\end{rem}

\end{document}